\documentclass[preprint,12pt]{elsarticle}
\usepackage{amsmath}
\usepackage{amssymb}
\usepackage{MnSymbol}
\usepackage{chngcntr}
\usepackage{caption}
\usepackage{subcaption}
\usepackage{graphicx}
\newtheorem{theorem}{Theorem}[section]
\newtheorem{remark}{Remark}[section]
\newtheorem{lemma}{Lemma}
\newenvironment{proof}[1][Proof]{\begin{trivlist}
\item[\hskip \labelsep {\bfseries #1}]}{\end{trivlist}}
\usepackage{epstopdf}
\usepackage{fancyhdr}
\journal{Chaos, Solitons and Fractals}
\counterwithout{equation}{section}
\begin{document}
\begin{frontmatter}
\title{A Dynamic Model for Infectious Diseases: The Role of Vaccination and Treatment}
\author[coauth]{P. Raja Sekhara Rao}
 \address[coauth]{Department of Mathematics, Government Polytechnic, Addanki, A.P., 523 201, India.}
\author[focal]{M.~Naresh Kumar\corref{cor1}}
 \ead{nareshkumar\_m@nrsc.gov.in}
 \cortext[cor1]{Principal Corresponding Author}
 \cortext[cor1]{Tel.: +91 40 2388 4388; Fax.: +91 40 2388 4437}
 \address[focal]{Software Group, National Remote Sensing Center (ISRO), Hyderabad, Telangana, 500 037, India.}
\fntext[fn1]{This article is dedicated to their research supervisor Prof. V. Sree Hari Rao for his continued inspiration.}
\markboth{\scriptsize Copyright $\copyright$ 2015 Elsevier B.V.. Accepted for Publication in Chaos, Solitons and Fractals.  DOI: 10.1016/j.chaos.2015.02.004} {\scriptsize \thepage}
\begin{abstract}
Understanding dynamics of an infectious disease helps in designing appropriate strategies for containing its spread in a population. Recent mathematical models are aimed at studying dynamics of some specific types of infectious diseases. In this paper we propose a new model for infectious diseases spread having susceptible, infected, and recovered populations and study its dynamics in presence of incubation delays and relapse of the disease. The influence of treatment and vaccination efforts on the spread of infection in presence of time delays are studied. Sufficient conditions for local stability of the equilibria and change of stability are derived in various cases. The problem of global stability is studied for an important special case of the model. Simulations carried out in this study brought out the importance of treatment rate in controlling the disease spread. It is observed that incubation delays have influence on the system even under enhanced vaccination. The present study has clearly brought out the fact that treatment rate in presence of time delays would contain the disease as compared to popular belief that eradication can only be done through vaccination.
\end{abstract}
\end{frontmatter}

\section{Introduction}
Infectious diseases spread through media such as air, water, direct contact or carriers such as insects, flies, and mosquitoes \cite{2,32}. The dynamics may well be understood by modeling causes or carriers of the disease. It is not possible all the time to control the media or carriers involved in spread of the infection and sometimes it may even be beyond determination. Many a time, we may not even recognize the presence of an infection until it becomes predominant. Mathematical modeling is a tool that has been popularly employed in prevention and control of infectious diseases such as severe acute respiratory syndrome (SARS) \cite{1}, human immunodeficiency virus infection/acquired immune deficiency syndrome (HIV/AIDS) \cite{4}, H5N1 (avian flu) \cite{33} and H1N1 (swine flu) \cite{31}. Also, they have been useful in studying some of the drug resistant strains of malaria \cite{12}, tuberculosis \cite{5}, methicillin-resistant staphylococcus aureus (MRSA) \cite{16} and marine bacteriophage infection \cite{3}.


Simple deterministic models are based on dividing the population into compartments such as susceptible, exposed, infected and recovered. The susceptible-Infective (SI) model is useful in understanding diseases such as feline infectious peritonitis (FIP) \cite{23} as the host remain infected till they die. The susceptible-infective-susceptible (SIS) model is employed in studying infections such as Gonorrhea \cite{27} as there is a possibility of host becoming susceptible to infection once again. The susceptible-infective-recovered (SIR) model is employed to understand diseases such as Syphilis \cite{29}, wherein, immunity lasts for a limited period before waning such that the individual is once again susceptible. Dynamics of such systems are generally studied from two perspectives, (i) stability of disease-free equilibrium, and (ii) stability of endemic equilibrium. In the former case conditions under which pathogens suffer extinction leaving individuals as susceptible are studied, whereas in the later, conditions in which  infective and susceptible populations coexist are established.

Interactions among populations in compartments such as the susceptible $(x)$ and infected $(y)$ is the key to understand the rate at which the infection can spread in the population. Models with bilinear interaction among populations $(xy)$ tend to disease free equilibrium or endemic equilibrium  \cite{7}. Researchers have considered various interactions of non-linear type such as $\beta {x^p} {y^q}$ \cite{19, 20, 28, 34, 43}, $\frac{xy}{1+ax}$ \cite{46}, or $\frac{xy}{1+y}$ \cite{13,15,30,37,38,39,45} for further understanding dynamics of infectious diseases. It is observed that solutions of periodic nature do exist even without periodic forcing term in the model. The existence of periodic solutions are usually attributed to the presence of time delays in the system even in case of bilinear incidence. Delays do exist in disease transmission, known as the latency period of infection, from inception to an identifiable state. For some interesting studies on infectious disease models, readers are referred to \cite{8,14,21a,14a,17,21,22,24,25,40,42,44}. On the other hand, fluctuations in populations may also be attributed to changes in environment, especially, in open systems. These perturbations are influenced by the environment and may be treated as noise in the system. Dynamic models of competing species under the influence of a noise are extensively studied in \cite{6,9,26,41}. Predictive models based on both clinical, incidence data pertaining to infectious diseases and epidemics in general are discussed in \cite{35,35a,35b}.


In the earlier studies on infectious disease models, the identification of basic reproduction number gave an idea on the stability of the system either in terms of disease free environment or disease prevalence. The reproduction number represents the number of secondary cases that are caused by a primary case when introduced into a wholly susceptible population. Also, it is observed that if reproduction number is less than unity, the disease free equilibrium is stable, otherwise, the disease prevails \cite{15}. When the nature of interaction is not known or the basic reproduction number is more than unity, one needs to find ways to restrict the disease to a minimum, manageable level, if the disease free environment cannot be provided. When it comes to the control of a disease, vaccination or preventive strategies could be a good choice during the early stages of the disease provided the population is not large enough. Emergence of vaccine resistant strains is one of the drawbacks of frequent or long term vaccination. When massive vaccination is not possible, the second stage of defensive mechanism could be medical treatment. This adds a new sub category of population called recovered class representing the population recovered by treatment.

Basing on the above observations, we consider the following three compartmental system in the present paper,
\begin{eqnarray}
\label{eqn1}
x' &=& a -b f(x,y)- d x - cV(x)+\alpha z \nonumber\\
y' &=& b_1 f\big(x(t-\tau), y(t)\big)- r P(y) -d_1 y \nonumber\\
z' &=& r P\big(y(t-\delta)\big) -\alpha z,
\end{eqnarray}
wherein, $x(t)$, $y(t)$ and $z(t)$ denote susceptible, infected and recovered (by treatment) populations at any time $t$ respectively, $'=\frac{d}{dt}$ denotes the time derivative of a function, $a \geq 0$ is the growth rate of susceptible population, $f$ denotes the non-linear incidence or infection function showing how susceptibles $x$ are converted into infected $y,$ $b>0$ denotes the rate of contact or interaction of infected with susceptible and $d$ is the rate of removal of such susceptible individuals from the system who are naturally immune to the infection and in no way get infected, $V(x)$ is the vaccination function (depends on susceptible population), $c>0$ is the rate of successful vaccination and $0< b_1 \leq b$ is the rate of conversion of susceptible into infected. 

In case, $b_1 =b,$ we may expect the infection to be at its helm and spreading $100\%$. If $b_1 <b,$ we may infer that the disease is not that effective or the missing $x'$s are no more susceptible and may be removed from the system. The time delay $\tau >0$ implies that $x$ when gets in contact with $y$ takes time `$\tau$' to become infected and is called the temporary immunity parameter. The parameter $d_1>0$ is the removal/death rate of the infected population-either not at all treated or inadequately treated or beyond the treatment. It may also be viewed as the rate at which infected population is kept in a quarantine, away from susceptible population avoiding any sort of contact between infected and exposed, $P(y)$ is the recovery (by treatment) function of the infected, $r>0$ denotes the rate of treatment and also the recovery rate. However, the infected by a treatment procedure may recover after a time gap $\delta.$ The parameter $\alpha >0$ is the rate at which a recovered (non-vaccinated) individual may be re-exposed to infection and become susceptible again. This case arises when the individual appears to be recovered by treatment but is prone to infection again. The basic interactions among various populations are shown in Figure \ref{Fig1}.

\begin{figure}[th]
  \includegraphics[height=70mm,width=\textwidth]{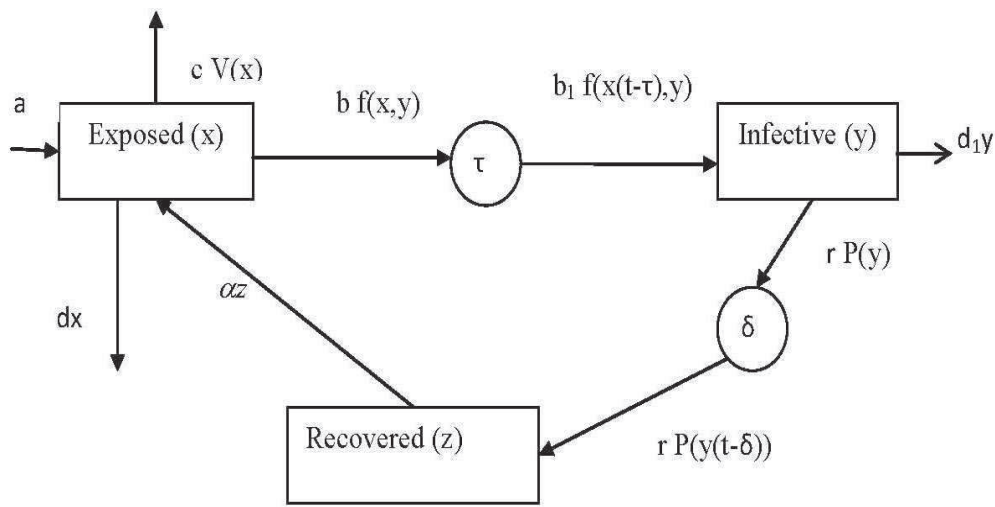}\\
  \caption{The flow dynamics of a general infectious disease model (\ref{eqn1})}\label{Fig1}
\end{figure}

We assume the following conditions on the infection function:
\begin{enumerate}[(i).]
\item $f(x,y) \geq 0$, $\forall x,\; y$
\item $f(0,y)=0,$ $\forall y,$ 
\item $f(x,0)=0 \forall x,$
\end{enumerate}
the first condition (non-negativity) is a minimum requirement for any function representing biological systems and for dynamics to make sense, the second condition arises out of the situation when there are no susceptible, i.e., $y$ has no more influence on them and the third condition is required to create a disease-free environment. Necessity for such condition is described in Section \ref{sec2}. 

In addition, a non-negativity condition is to be satisfied by $P$ as it represents the recovery (by treatment) function. Thus, we assume $P(y) \geq 0,$ $\forall y$ with strict inequality holding for $y>0,$ as we are interested in studying the influence of treatment on spread of disease. In the absence of infected population, we have $P(0)=0$ as no treatment effort is required.

Depending on the reproduction number of an epidemic and its availability, we need to initiate vaccination. The following are the probable cases, (i) a uniform constant supply $V(x)= k,$ (ii) a fixed constant or proportional to susceptible population $V(x)= x$ (linear), (iii) sub-linear $V(x)= \frac{x}{k+x}$ or $V(x)=\tanh{x}$ (have a saturation limit), and (iv) $V(x)$ may be a periodic function when the infection is fluctuating, season-dependent or cyclic. However, in case of a new infectious disease, system may not be ready with vaccines to prevent its spread, the recent case being the spread of Ebola \cite{16a}. Thus, a time gap is required to start the vaccination procedure until a vaccine is invented and supplied. So, we may have $V(x(t))=0,$ $\forall t$, such that, $0\leq t \leq t_1$ and $V=V(x(t)) \geq 0,$ $\forall t > t_1.$ A delay is also warranted in this section which we shall discuss in the final section.

A basic motivation for this study stems from observations in \cite{37} that vaccination alone is not sufficient to contain disease effectively and treatment plays an important role in the disease management. We shall analyze the model (\ref{eqn1}) keeping in view the following concerns which are important in handling infectious disease:
\begin{enumerate}
  \item The influence of treatment or recovery rate \emph{r},
  \item Influence of vaccination effort $V,$
  \item Combined effect of vaccination and treatment,
   \item How to destabilize equilibrium $(x^*,y^*,z^*)$ when $y^* >0$ ? or can we make $y \to 0.$
 \end{enumerate}


The present paper aims to provide answers to the above questions by organizing its content as follows: in Section \ref{sec2}, the possible equilibria are determined for the model equations (\ref{eqn1}). In Section \ref{sec3}, local stability properties of the system are discussed. In Subsection \ref{subsec31}, the behaviour of delay-free system is discussed. Influence of time delay $\tau>0$ on the stability of the delay-free system is studied in Subsection \ref{subsec32}, while the influence of delay in recovery ($\delta >0$) alone is studied in Subsection \ref{subsec33}. Bounds for these delays, conditions for change in the direction of stability are also discussed. In Subsection \ref{subsec34}, the general case of delays $\tau>0$ and $\delta>0$ is discussed. In Section \ref{sec4}, global stability results are derived for a special case in which both the delays are present. Examples and simulation results are presented in Section \ref{sec5}, while discussions and scope for future study are deferred to Section \ref{sec6} and Section \ref{sec7} respectively.

\section{Equilibria and Characteristic Equation}
\label{sec2}
Let us consider a case of the system (\ref{eqn1}) with assumptions made above as
\begin{eqnarray}
\label{eqn21}
x' &=& a -b f(x,y)- d x - c V(x)+\alpha z \nonumber\\
y' &=& b_1 f(x(t-\tau), y) - r P(y) -d_1 y \nonumber\\
z' &=& r P(y(t-\delta)) -\alpha z.
\end{eqnarray}

By continuity of solutions of (\ref{eqn21}) and by definitions of functions  $f,\; V$ and $P,$ it is easy to see that $x,y,z$ are all non-negative in their domains of definition.

\subsection{\textbf{Equilibria}}
\label{subsec21}
Equilibria are the constant solutions of the system. A popular way to understand the system is to study the behaviour of its equilibria. For (\ref{eqn21}), we look for two possibilities, (i) existence of an equilibrium of the type $(\bar{x},0,0)$ called disease-free equilibrium whose stability implies that the infection has little influence, and (ii) a stable positive equilibrium $(\tilde{x},\tilde{y},\tilde{z})$ called endemic equilibrium, which denotes the disease prevalence. Unless specifically stated, $(x^*,y^*,z^*)$ denotes either of these equilibria throughout the subsequent study.

The equilibria $(x^*, y^*, z^*)$ of (\ref{eqn21}) should satisfy
\begin{eqnarray}
b f(x^{*},y^{*})+c V(x^{*})+d x^{*}- \alpha z^{*} &=& a  \nonumber\\
        b_1 f(x^{*},y^{*})-d_1 y^{*}-r P(y^{*}) &=& 0 \nonumber\\
        r P(y^{*}) &=& \alpha z^{*}.
\end{eqnarray}

We hypothesize that if $0<V(\frac{a}{d})< \infty$ holds, then there shall exist a point $\bar{x}$ such that it satisfies the condition $cV(\bar{x}) = a - d\bar{x}$. To further illustrate our hypothesis we considered a real-time dataset pertaining to SARS \cite{23a} and fitted a logistic curve of the form $V(x)=\frac{C}{1+A \exp^{-Bx}}$ to the number of reported SARS cases ($x$) in a given time interval. We estimated parameters of the logistic curve from data as $C=206$, $B=0.1542$ and $A=102$. The inflection point of the logistic curve give by $(\frac{\ln{A}}{B},\frac{C}{2})$ is the point $(30,103)$ shown in Figure \ref{Fig2} at which the curve switches from an increasing rate of growth to a decreasing rate, clearly suggesting the existence of a disease-free equilibrium point $(\bar{x},0,0)$ whose stability implies a disease-free environment. The logistic curve and the line passing through the disease-free equilibrium point $\bar{x}$ is shown in Figure \ref{Fig2}. This provides a mechanism to design the vaccination ($V$) methods for containing the disease. 

\begin{figure}[th]
	\includegraphics[height=60mm,width=\textwidth]{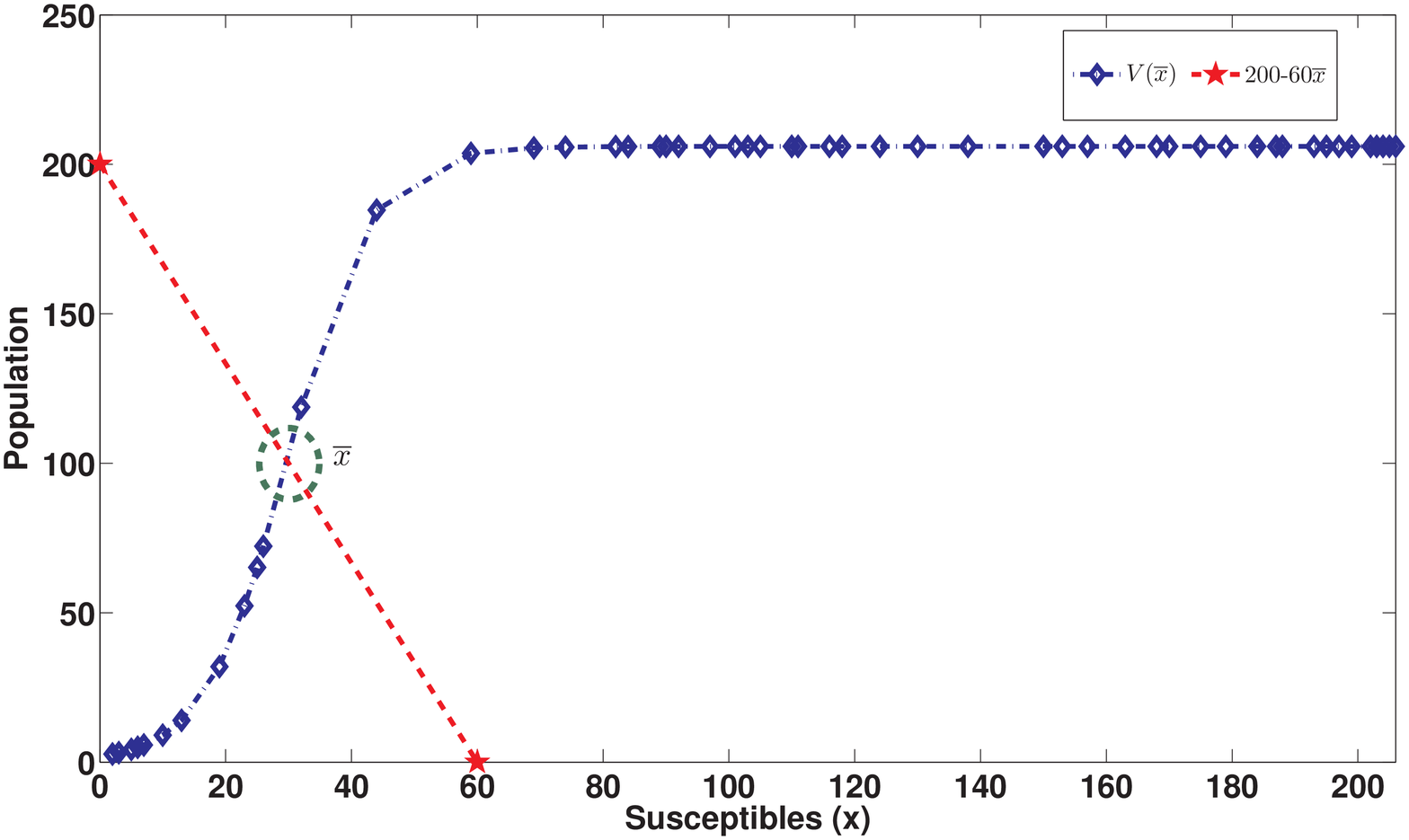}
  \caption{Disease-free equilibrium point}
	\label{Fig2}
\end{figure}

 \subsection{\textbf{Characteristic Equation}}

Linearizing the system (\ref{eqn21}) around the equilibrium $(x^{*},y^{*},z^{*})$, we get

 \[\left[ \begin{array}{l}
 x^{'}  \\
 y^{'}  \\
 z^{'}  \\
 \end{array} \right] = \left[ {\begin{array}{*{10}c}
   { - b\frac{{\partial f}}{{\partial x}} - cV^{'}\left(x \right) - d} & { - b\frac{{\partial f}}{{\partial y}}} & {\alpha }  \\
   0 & {b_1 \frac{{\partial f}}{{\partial y}} - d_{1} - rP^{'}\left(y \right)} & 0  \\
   0 & 0 & { - \alpha }  \\
\end{array}} \right]\left[ \begin{array}{l}
 x \\
 y \\
 z \\
 \end{array} \right]\]\quad \[+ \left[ {\begin{array}{*{10}c}
   0 & 0 & 0  \\
   {b_1 \frac{{\partial f\left( {x_\tau  ,y} \right)}}{{\partial x_\tau  }}} & 0 & 0  \\
   0 & {r\frac{{\partial P}}{{\partial y_\delta  }}} & 0  \\
\end{array}} \right]\left[ \begin{array}{l}
 x_\tau   \\
 y_\delta   \\
 0 \\
 \end{array} \right]\]

or \\
\begin{equation}
\label{eqn23}
 \left[ \begin{array}{l}
 x^{'}  \\
 y^{'}  \\
 z^{'}  \\
 \end{array} \right] = \left[{\begin{array}{*{10}c}
  A & B & { \alpha }  \\
   0 & C & 0  \\
   0 & 0 & { - \alpha }  \\
\end{array}} \right]\left[ {\begin{array}{l}
 x \\
 y \\
 z \\
 \end{array}} \right] + \left[ {\begin{array}{*{10}c}
 0 & 0 & 0  \\
   D & 0 & 0  \\
   0 & E & 0  \\
\end{array}} \right]\left[ \begin{array}{l}
 x_\tau   \\
 y_\delta   \\
 0 \\
 \end{array} \right]
\end{equation}
\\
where $A={- b\frac{{\partial f}}{{\partial x}} - cV^{'}\left(x \right) - d}$, $B={ - b\frac{{\partial f}}{{\partial y}}}$, $C={b_1
\frac{{\partial f}}{{\partial y}} - d_{1} - rP^{'}\left(y \right)}$, \\ $D=b_{1}\frac{{\partial f\left( {x_\tau  ,y} \right)}}{{\partial x_\tau }}$ $,$ $E={r\frac{{\partial P}}{{\partial y_\delta  }}}$ are evaluated at $(x^{*},y^{*},z^{*}),$ where $x_{\tau} = x(t-\tau),\; y_{\delta} = y(t-\delta).$

The characteristic equation of (\ref{eqn23}) is given by 
\[
F(\lambda)
 = \left |{\begin{array}{*{10}c}
  {(A-\lambda)} & B & { \alpha }  \\
   {De^{-\lambda\tau}} & {C-\lambda} & 0  \\
   0 & Ee^{-\lambda\delta} & { - \alpha-\lambda }  \\
\end{array}} \right |=0.
\]

That is,
\begin{eqnarray}
\label{eqn24}
 F(\lambda)\equiv
\lambda^{3}-(A+C-\alpha)\lambda^{2}+[AC-\alpha(A+C)]\lambda+\\ \nonumber
\alpha CA-B(\alpha+\lambda)De^{-\lambda\tau}-D\alpha Ee^{-\lambda(\tau+\delta)}=0.
\end{eqnarray}
Letting $ l=\alpha-(A+C),\; m=AC-\alpha(A+C),\; n=\alpha CA, \; l_{1}=-BD,\; m_{1}=-B\alpha D,\; n_{1}=-D\alpha E$ we denote (\ref{eqn24}) as
\begin{equation}
\label{eqn25}
F(\lambda)\equiv \lambda^{3}+l\lambda^{2}+m\lambda+n+l_{1}\lambda e^{-\lambda\tau}+m_{1}e^{-\lambda\tau}+n_{1}e^{-\lambda(\tau+\delta)} =0.
\end{equation}
Letting $\lambda=\mu+i\nu$ in (\ref{eqn25}) and separating the real and imaginary parts of $F(\mu+i\nu)$, we have real and imaginary parts (denoted hereafter by $\Re(.)$ and $\Im(.)$) of $F(\lambda)$ as
\begin{eqnarray}
\label{eqn26}
\Re(F(\lambda)) &=& \mu^{3}-3\mu\nu^{2}+l(\mu^{2}-\nu^{2})+m\mu+n +l_{1}[\mu\cos\nu\tau+ \nonumber \\ 
                         &&\nu\sin\nu\tau]e^{-\mu\tau}+m_{1}e^{-\mu\tau}\cos\nu\tau+ n_{1}\cos\nu(\tau+\delta)e^{-\mu(\tau+\delta)} \nonumber \\
\Im(F(\lambda)) &=& -\nu^{3}+3\mu^{2}\nu+2l\mu\nu+m\nu+l_{1}[\nu\cos\nu\tau-  \nonumber \\ 
&&\mu\sin\nu\tau]e^{-\mu\tau}-m_{1}e^{-\mu\tau}\sin\nu\tau-n_{1}\sin\nu(\tau+\delta)e^{-\mu(\tau+\delta)}.
\end{eqnarray}

\section{Local Stability}
\label{sec3}
\subsection{\textbf{Delay-Free System}}
\label{subsec31}
We shall first consider the case when there are no delays in contracting the disease ($\tau=0$) and in recovery
($\delta=0$). The equation (\ref{eqn25}) reduces to $$\lambda^{3}+ l \lambda^{2}+ (l_1 +m)\lambda+ n+m_1 +n_1=0.$$ The elementary algebra says that if (i) $l>0,$ $l_1+m>0,$ and (ii) $n+m_1+n_1>0$ hold, all the roots of the above equation have negative real parts. Thus, we have the Theorem \ref{th31} which needs no proof.

\begin{theorem}
\label{th31}
For $\tau=0$, $\delta=0$, the system (\ref{eqn21}) is locally asymptotically stable provided that the parameters satisfy $l>0,$ $l_1+m>0$ and $n+m_1+n_1>0.$ $\square$
\end{theorem}

\begin{remark}
\label{th32}
In particular, if $D=0,$ $C\leq 0$ hold, and when $V'(x) \geq 0$, $\frac{\partial f}{\partial x} \geq 0,$ we have $A\leq 0$then the characteristic equation (\ref{eqn25}) reduces to $$\lambda^{3}+l\lambda^{2}+m\lambda+n=0,$$ which is delay free. Clearly, all its coefficients are non-negative and there is no possibility of a root with positive real part. Thus, the equilibrium $(x^*,y^*,z^*)$ is stable when $D=0,$ $C\leq 0$ and $A\leq 0.$ Thus, we need to explore possibilities of finding a disease-free equilibrium at which these conditions hold - implying the possibility of disease-free environment. The choice of parameters, selection of vaccination function, identification of infection function and recovery may render this possible.  $\square$
\end{remark}

We may further infer the following:
\begin{enumerate}
\item In case the system has both a disease free equilibrium $E_0$ and an endemic equilibrium $E_1$ and if conditions of Theorem \ref{th31} hold at $E_0$ then we may say that the disease has no influence,
\item If $E_0$ is either unstable (at least one of the roots of the above equation has a positive real part) or does not exist at all, then we check for the stability of the endemic equilibrium $E_1,$
\item If $E_1$ is stable, we have to explore the possibilities of destabilizing it by increasing (i) vaccination efforts, (ii) treatment rate or by (iii) delay in process of infection,
\item In case $E_1$ is unstable but the solutions are exhibiting the dominance of the disease  or the system has no $E_1$, we need to study the behaviour of solutions in a way to handle the infected population (such as finding an $\epsilon >0,$ $T>0$ such that $y(t)\leq \epsilon $ for $t>T$).
\end{enumerate}

We shall provide some answers to these towards the end of this article. We shall now study the influence of time delays on the system. First, we start with a delay in incubation but no delay in the recovery of treated.

\subsection{\textbf{Delay only in Incubation (parameter $\tau>0,\; \delta=0$)}}
\label{subsec32}
To understand the influence of the time delay in converting the exposed population to infected only, we allow $\delta =0$ in equation (\ref{eqn25}). We get
\begin{equation}
\label{eqn31}
\lambda^{3}+l\lambda^{2}+m\lambda+(l_{1}\lambda+m_{1}+n_{1})e^{-\lambda\tau}+n = 0.
\end{equation}
Equations of the type (\ref{eqn31}) are well studied in literature (e.g., \cite{8,10,11}). Letting $a_0 = n^2 -(m_1+n_1)^2,$ $ a_1 =m^2 -2ln-{l_1}^2$ and $a_2 =l^2 -2m,$ a straight forward application of Theorem 4.1 of \cite{8} yields the Theorem \ref{th33}.

\begin{theorem}
\label{th33}
Assume that either (i) $a_0 >0,$ $a_1 \geq 0$ or (ii) $a_0 >0,$ $a_1 <0,$ and $2a_2^3 -9a_1a_2 + 27a_0 > 2(a_2^2-3a_1)^{3/2}$ holds. Then if the equilibrium $(x^*,y^*,z^*)$ of (\ref{eqn21}) is stable(unstable) at $\tau =0,$ it remains stable(unstable) for any $\tau >0.$ If $a_0 \leq 0$ then if $(x^*,y^*,z^*)$ is unstable for $\tau =\tau_0 \geq 0$ then it will be unstable for all $\tau >\tau_0.$ If $a_0 <0$ and if $(x^*,y^*,z^*)$ is stable at $\tau=0$ then there exists a $\tilde{\tau} >0$ for which $(x^*,y^*,z^*)$ is unstable for $\tau >\tilde{\tau}.$ $\square$
\end{theorem}

Notice that Theorem \ref{th33} (or Theorem 4.1 of \cite{8}) ensures the preservation of stability (instability) but does not specify the values of lengths of delay for which this holds. The following result helps us in estimating delay in such cases. We need the following result (\cite{18}) as utilized in \cite{36} for our study.

\begin{lemma} \textbf{(Theorem 1, \cite{18})}
\label{lm1}
A system with characteristic equation \\ $F(\lambda, \tau) \equiv  P(\lambda) + Q(\lambda) e^{-\lambda \tau } =0$ is asymptotically stable if and only if the following conditions are satisfied, \begin{enumerate}[(i).] \item \label{lm11} $F_1(\lambda) = P(\lambda) + Q(\lambda) \neq 0,\; \Re(\lambda) \geq 0$, \item \label{lm12} $F_2(\lambda) = P(\lambda)- Q(\lambda) \neq 0,  \Re(\lambda) = 0, \lambda \neq 0$, \item \label{lm13} $F_3(\lambda) = P(\lambda) + \frac{1-\lambda T}{1+\lambda T} Q(\lambda) \neq 0, \; \Re(\lambda) = 0, \; \lambda \neq 0, \; \forall \,\, 0<T<\infty$ \end{enumerate}

where, $T$ denotes the pseudo delay and is related to $\tau$ as $\tau=\frac{2}{\nu}[\tan^{-1}(\nu T)+k\pi $], $\forall k \in \mathbb{Z}$. $\square$
\end{lemma}

The conditions (\ref{lm11}) in Lemma \ref{lm1} ensures stability for $\tau =0,$ (\ref{lm12}) confirms that no pure imaginary zeros in the limiting case and condition (\ref{lm13}) establishes the absence of imaginary zeros. It may be observed that $F_1,\; F_2,\; F_3$ are obtained from $F(\lambda, \tau)$ by letting $e^{-\lambda \tau} = 1,\; -1,\; (1-\lambda T)/(1+\lambda T)$ respectively. The Lemma \ref{lm1} is a necessary and sufficient condition and violation of any of the conditions (\ref{lm11}), (\ref{lm12}) or (\ref{lm13}) reflects corresponding change in stability of characteristic equation (\cite{20}).

Employing Lemma \ref{lm1}, we shall present another set of conditions for preservation or change of stability of the system. To apply above lemma, we shall denote (\ref{eqn31}) as
\begin{equation}
\label{eqn32}
 P(\lambda)+Q(\lambda)e^{-\lambda\tau}=0,
\end{equation} 
in which $P(\lambda)=\lambda^{3}+l\lambda^{2}+m\lambda+n$ and $Q(\lambda)=l_{1}\lambda+m_{1}+n_{1}.$ 
\begin{enumerate}

\item Let $e^{-\lambda\tau}=1$ in (\ref{eqn32}). The characteristic polynomial becomes $$F_{1}(\lambda)\equiv P(\lambda)+Q(\lambda)=\lambda^{3}+l\lambda^{2}+(l_{1}+m)\lambda+(m_{1}+n_{1}+n).$$ clearly $F_{1}(\lambda)\neq 0$ for $\Re(\lambda)\geq 0$, by our Theorem \ref{th31}. This ensures stability for $\tau=0.$

\item Let $e^{-\lambda\tau}=-1$ in (\ref{eqn32}) to get $$F_{2}(\lambda)\equiv P(\lambda)-Q(\lambda)=\lambda^{3}+l\lambda^{2}+(m-l_{1})\lambda+(n-m_{1}-n_{1}).$$ Suppose $\lambda=i\nu$ then $$F_{2}(\lambda)=(i\nu)^{3}+l(i\nu)^{2}+(m-l_{1})(i\nu)+(n-m_{1}-n_{1}) $$ $$= i[-\nu^{2}+m-l_{1}]\nu+((n-m_{1}-n_{1})-l\nu^{2}).$$ Clearly $F_{2}(\lambda)= 0 $ if and only if a $\nu$ exists such that ${\nu}^2 = m-l_1 = \frac{n-m_1 -n_1}{l}$ holds. Clearly $F_{2}(\lambda)\neq 0 $ for $\Re(\lambda)=0$, $\lambda\neq 0$ provided any of the following conditions

\begin{enumerate}
\item $m-l_1\leq 0$ or
\item  $l(n-m_{1}-n_{1})\leq 0$ or $l(m-l_1) \ne n-m_1 -n_1$ 
\end{enumerate}
holds, ensuring the presence of no imaginary zeros in the limiting case.

\item Let $e^{-\lambda\tau}=\frac{1-\lambda T}{1+\lambda T}$ from some $T>0.$ Then (\ref{eqn32}) becomes, $$F_{3}(\lambda)\equiv P(\lambda)+\frac{1-\lambda T}{1+\lambda T}Q(\lambda).$$ If we can show that $F_{3}(\lambda)\neq 0$ for  $\Re(\lambda)=0$, $\lambda\neq 0$,$ \forall \; 0<T<\infty$ then (\ref{eqn32}) has no pure imaginary zeros, ensuring no change of stability. 
\end{enumerate}

Now, 
\begin{eqnarray*}
F_{3}(\lambda)&=& \lambda^{3}+l\lambda^{2}+m\lambda+n+(l_{1}\lambda+m_{1}+n_{1})\frac{1-\lambda
T}{1+\lambda T} \\
&=& \frac{\lambda^{4}
T+(lT+1)\lambda^{3}+(l+mT-l_{1}T)\lambda^{2}}{1+\lambda T}+ \\
&&\frac{(m+nT-m_{1}T-n_{1}T+l_{1})\lambda+(n+m_{1}+n_{1})}{1+\lambda T}.
\end{eqnarray*}

Further, $F_{3}(\lambda)=0$ iff $$\lambda^{4} T+(lT+1)\lambda^{3}+(l+mT-l_{1}T)\lambda^{2}+(m+nT-m_{1}T-n_{1}T+l_{1})\lambda+(n+m_{1}+n_{1})=0.$$

Assume that $n+m_{1}+n_{1}\neq 0,$ otherwise $\lambda=0.$ Letting $\lambda=i\nu$ in $F_{3}(\lambda)=0$ and separating the real and imaginary parts, we get real part as \\
$\nu^{4}+((l_1-m)T-l)\nu^{2}+(n+m_{1}+n_{1})=0,$ \\
and imaginary part as \\
$-(lT+1)\nu^{3}+(m+nT-m_{1}T-n_{1}T+l_{1})\nu=0.$ 

From the imaginary part, we have $\nu\neq 0 \Rightarrow \nu^{2}=\frac{(l_{1}+m)+(n-m_{1}-n_{1})T}{1+lT}.$ The following cases arise:
\begin{enumerate}
\item no real $\nu$ satisfies this if $\frac{(l_{1}+m)+(n-m_{1}-n_{1})T}{1+lT} <0$ holds for any $T>0.$
This after a rearrangement gives $l(n-m_1 -n_1) T^2 + [l(l_1+m) + (n-m_1 -n_1)]T + (l_1 +m) <0.$ Thus, if $l(n-m_1 -n_1)>0,\; [l(l_1+m) + (n-m_1 -n_1)]>0$ and $(l_1 +m) >0,$ no real $\nu$ exists, thus, establishing the no change of stability for $\tau\geq 0$ given earlier. 

\item using the value of $\nu^{2}$ from imaginary part in real part, we get $$[\frac{(l_{1}+m)+(n-m_{1}-n_{1})T}{1+lT}]^{2}$$ $$-[\frac{(l_{1}+m)+(n-m_{1}-n_{1})T}{1+lT}](l+mT-l_{1}T)+(n+m_{1}+n_{1})=0.$$ This after a rearrangement becomes
\begin{eqnarray*}
&-& l(n-m_{1}-n_{1})(m-l_{1})T^{3} \\
&+& [(n-m_{1}-n_{1})^{2}\\
&&-[(m^{2}-l_{1}^{2})l+l^{2}(n-m_{1}-n_{1})-(n-m_{1}-n_{1})(m-l_{1})]\\
&+& l^{2}(n+m_{1}+n_{1})]T^{2}+ [2(l_{1}+m)(n-m_{1}-n_{1})\\
&&-[(l_{1}+m)l^{2}+l(n-m_{1}-n_{1})+(m^{2}-l_{1}^{2})]+2l(n+m_{1}+n_{1})]T \\
&+& [(l_{1}+m)^{2}-l(l_{1}+m)+(n+m_{1}+n_{1})]=0.
\end{eqnarray*}
This may be written as
\begin{equation}
\label{eqn33}
\overline{A}T^{3}+\overline{B}T^{2}+\overline{C}T+\overline{D}=0.
\end{equation}
\end{enumerate}

We summarize the above discussion in Theorem \ref{th34} as
\begin{theorem}
\label{th34}
Assume that the equilibrium solution of (\ref{eqn21}) is stable for $\tau=0$ and $\delta=0.$ Then the system preserves its stability for any length of delay $\tau >0$ and $\delta=0$ provided either of the following cases satisfy
\begin{enumerate}
\item $l(n-m_1-n_1) >0,\; l(l_1+m) + (n-m_1-n_1) >0,\; l_1+m >0,$
\item $\overline{A}\geq 0, \; \overline{B}\geq 0, \; \overline{C}\geq 0, \; \overline{D}\geq 0$ hold,
\item a change of stability occurs at $\tau =\tau_+ >0$ if any of the coefficients $\overline{A}, \;\overline{B}, \; \overline{C}, \;\overline{D}$ is negative, 
\end{enumerate}
here, $\tau_+ = \frac{2}{\nu_+} tan^{-1} ({\nu_{+} T_{+}}), \, \nu_+ = \frac{l_1+m + (n-m_1-n_1)T_{+}}{1+lT_{+}} $ and $T_{+}>0$  \\  is a solution of (\ref{eqn33}). $\square$
\end{theorem}

We shall now study the influence of time delay only in recovery on the stability of the system. We wish to understand clearly whether the time delays $\tau$ and $\delta$ influence independently or are they introduced out of mathematical curiosity. To be specific, if $\delta$ has no influence on the system, better we ignore it and proceed with $\tau$ only.

\subsection{\textbf{Case in which the parameters $\tau=0$ and $\delta>0$}}
\label{subsec33}
In this case the characteristic equation (\ref{eqn25}) reduces to 
\begin{equation}
\label{eqn34}
F(\lambda)\equiv \lambda^3 + l\lambda^2 + (l_1 +m) \lambda +n+m_1 + n_1 e^{-\lambda\delta} =0.
\end{equation}

We shall utilize Lemma \ref{lm1} again to analyze the stability behaviour of the system in this case also. The following cases may arise
\begin{enumerate}

\item Let  $e^{-\lambda\delta}=1$ in (\ref{eqn34}) to get $$F_1(\lambda)\equiv \lambda^3 + l\lambda^2 + (l_1 +m) \lambda +n+m_1 + n_1 =0.$$ Recalling Theorem \ref{th31}, $F_1(\lambda)\ne 0$ for $\Re(\lambda) \geq 0.$ 

\item Let $e^{-\lambda\delta}=-1$ in (\ref{eqn34}) to get $$F_2(\lambda)\equiv \lambda^3 + l\lambda^2 + (l_1 +m) \lambda +n+m_1 - n_1 =0.$$ Letting $\lambda = i\nu,$ and separating real and imaginary parts, we get $$F_2(i\nu) = n+m_1 - n_1  -l\nu^2 + (l_1+m-\nu^2)\nu i.$$ It is easy to see that $F_2(i\nu) \ne 0$ if
\begin{equation}\label{eqn35} l(l_1+m) \ne n+m_1 -n_1.\end{equation} This condition ensures no pure imaginary zeros in limiting case.

\item Consider $e^{-\lambda\delta}=\frac{1-\lambda T}{1+\lambda T}.$ Then (\ref{eqn34}) gives $$F_3(\lambda)\equiv \lambda^3 + l\lambda^2 + (l_1 +m) \lambda +n+m_1 + n_1 \frac{1-\lambda T}{1+\lambda T}.$$Now $F_3(\lambda) =0$ if and only if $$[\lambda^3 + l\lambda^2 + (l_1 +m) \lambda +n+m_1 - n_1]\lambda T + \lambda^3 + l\lambda^2 + (l_1 +m) \lambda +n+m_1 + n_1 =0.$$ For pure imaginary zeros, we let $\lambda =i\nu$ in the above. Separating real and imaginary parts, we get real part as $$[\nu^4 -(l_1+m)\nu^2]T -l\nu^2 +n+m_1+n_1 =0.$$ and imaginary part as $$ [-l\nu^3 +(n+m_1-n_1)\nu]T + [-\nu^3 + (l_1+m)\nu]=0.$$

Eliminating $T$ and rearranging we get for $\nu \ne 0$

\begin{equation}
\label{eqn36}
\nu^6 + (l^2 -2(l_1+m))\nu^4 + [(l_1+m)^2-2l(n+m_1)]\nu^2 +(n+m_1)^2-{n_1}^2 =0.
\end{equation}
If the inequalities
\begin{equation}
\label{eqn37}
l^2 -2(l_1+m) \geq 0,\;\; (l_1+m)^2-2l(n+m_1)\geq 0 \;\; \mbox{and} \;\; (n+m_1)^2-{n_1}^2 \geq 0
\end{equation}
 hold then (\ref{eqn36}) yields no real solution for $\nu.$ Thus, the stability/instability of (\ref{eqn34}) is preserved for any length of delay $\delta.$

\end{enumerate}

\subsubsection{\textbf{Change of stability}}
The following cases may arise:
\begin{enumerate}
\item If any of the following conditions
\begin{eqnarray}
\label{eqn38}
&(i).& l^2 < 2(l_1+m)), (l_1+m)^2 > 2l(n+m_1), (n+m_1)^2 > {n_1}^2 \nonumber \\
&(ii).& l^2 > 2(l_1+m)), (l_1+m)^2 < 2l(n+m_1), (n+m_1)^2 > {n_1}^2 \nonumber \\
&(iii).& l^2 > 2(l_1+m)), (l_1+m)^2 > 2l(n+m_1), (n+m_1)^2 < {n_1}^2 \nonumber \\
&(iv).& l^2 < 2(l_1+m)), (l_1+m)^2 < 2l(n+m_1), (n+m_1)^2 < {n_1}^2 \nonumber \\ 
\end{eqnarray}
hold, a change of sign indicates that (\ref{eqn36}) has one $\nu^2 >0$ yielding two possibilities for $\nu.$ For these values of $\nu,$ we have
\begin{eqnarray}
\label{eqn39}
&T &= \frac{\nu^2 -(l_1+m)}{n+m_1-n_1-l\nu^2} \\ \nonumber
&\delta& =\frac{2}{\nu} [tan^{-1} \nu T +k\pi], \forall k \in \mathbb{Z}.
\end{eqnarray}

Certainly, $T>0$ exists if $\nu^2$ lies between $l_1+m$ and $\frac{n+m_1-n_1}{l}$ and at least one of these terms is positive. This gives rise to a change of stability.

\item If conditions
\begin{equation} \label{eqn310} l^2 < 2(l_1+m)), \;\; (l_1+m)^2 > 2l(n+m_1)\;\; \mbox{and} \;\; (n+m_1)^2 < {n_1}^2 \end{equation}hold, (\ref{eqn36}) yields two values for $\nu^2>0$ and there will be four real values for $\nu$. Again from (\ref{eqn39}) we get two values each for $T$ say $T_1$ and $T_2$ and $\delta$ say $\delta_1$ and $\delta_2.$ If $0 <\delta_1 < \delta_2$ holds, then a change in stability occurs first at $\delta_1$ and a second bifurcation at $\delta_2.$
\end{enumerate}

Similar arguments extend to the three values of $\nu^2>0.$ We consolidate the above discussion as Theorem \ref{th35}
\begin{theorem} 
\label{th35}
\begin{enumerate}[(i).]
\item The stability/instability of $(x^*,y^*,z^*)$ is preserved for any length of delay $\delta$ if the conditions (\ref{eqn35}) and conditions (\ref{eqn37}) together hold. \\
\item If any of the conditions (\ref{eqn38}) holds then there exists a $\tilde \nu $ where ${\tilde{\nu}}^2$ lies between $l_1+m$ and $\frac{n+m_1-n_1}{l}$ satisfying (\ref{eqn36}) and there is a change in stability for $\delta =\frac{2}{\nu} tan^{-1} \tilde{\nu} \tilde{T}$, where $\tilde{T}=\frac{\tilde{\nu^2} -(l_1+m)}{n+m_1-n_1-l\tilde{\nu}^2} >0.$ \\
\item If (\ref{eqn310}) holds then there is a possibility of a second bifurcation.$\square$
\end{enumerate}
\end{theorem}

\subsection{\textbf{General Case in which the parameters $\tau>0,$ $\delta>0$}}
\label{subsec34}
We now consider the characteristic equation of the system (\ref{eqn21}) given by (\ref{eqn25}). The condition for a change in stability (occurrence of Hopf bifurcation) is indicated by a zero of the real part of $\lambda=\mu+i\nu$ of (\ref{eqn25}). Letting $\mu=0$ in (\ref{eqn26}) we get $$\Re(F(i\nu))= -l\nu^{2}+n+n_{1}\cos\nu(\tau+\delta)+l_{1}\nu\sin\nu\tau+m_{1}cos\nu\tau.$$ $$\Im(F(i\nu))=-\nu^{3}+m\nu-n_{1}\sin\nu(\tau+\delta)+l_{1}\nu\cos\nu\tau+m_{1}\sin\nu\tau.$$ Squaring and adding, we get
\begin{eqnarray}
\label{eqn311}
\Psi(\nu) &\equiv & \nu^{6}+(l^{2}-2m)\nu^{4}+(m^{2}-2ln-l_{1}^{2})\nu^{2}+(n^{2}+n_{1}^{2}-m_{1}^{2})
\nonumber\\
&&+ 2(n-l\nu^{2})\cos\nu(\tau+\delta)+2(m\nu-\nu^{3})\sin\nu(\tau+\delta).
\end{eqnarray}

We establish the following result.
\begin{theorem}
\label{th36}
Assume that the equilibrium solution $(x^*,y^*,z^*)$ of (\ref{eqn21}) is stable (unstable) for $\tau=0$ and $\delta =0.$ \\
\begin{enumerate}[(i).]
\item The stability (instability) of $(x^*,y^*,z^*)$ is preserved for any $\tau>0$ and $\delta >0$ provided the conditions
\begin{equation}
\label{eqn312}
\left( \frac{|l_1|+\sqrt{l_1^2 +4l(n+|n_1|+|m_1|)}}{2l} \right)^2 \leq \frac{n_1^2-(m_1^2 +2)}{l_1^2} \; \mbox{and} \; n_1^2 > m_1^2 +2
\end{equation}
hold. \\
\item A change of stability or instability occurs if the condition $n^2 +2n +n_1^2 < m_1^2$ holds.
\end{enumerate}
\end{theorem}
\begin{proof}[Proof of Theorem \ref{th36}] 

The conditions for preservation of stability or instability are given by $\Re(F(i\nu))=0$ and $\Im(F(i\nu)) >0$ at any $\nu =\nu_0$ at $\mu=0.$ Equivalently, if we can show that $\Psi(\nu_0) >0$ then the stability/instability is preserved. Without loss of generality, let $\nu_0$ be the smallest positive root of $\Re(F(i\nu))=0.$ Since,  $n_{1}\cos\nu(\tau+\delta)\leq |n_1|,$ $l_{1}\nu\sin\nu\tau \leq |l_1|\nu$ and $m_{1}cos\nu\tau \leq |m_1|$ for any positive $\nu$ we have $l{\nu}^2 -n  \leq |n_1| + |l_1|\nu + |m_1|,$ from $\Re(F(i\nu)).$

Consider $l{\nu}^2 -|l_1|\nu- (n+ |n_1|+ |m_1|)=0.$ This has clearly two roots and the bigger one is given by ${\nu}_+ = \frac{|l_1|+ \sqrt{l_1^2 +4l(n+|n_1|+|m_1|)} }{2l}.$ Then clearly ${\nu}_0 \leq {\nu}_+.$ Now from (3.11), we have $\Psi(\nu) >0$ is implied if $$\nu^{6}+(l^{2}-2m)\nu^{4}+(m^{2}-2ln-l_{1}^{2})\nu^{2}+(n^{2}+n_{1}^{2}-m_{1}^{2})> 2|n-l\nu^{2}|+2\nu|m-\nu^{2}|$$ holds i.e., if $$(|n-l\nu^{2}| -1)^2+ (\nu|m-\nu^{2}| -1)^2 > l_1^2 {\nu_0}^2 -n_1^2 +m_1^2 +2 $$ holds.

The left hand side is  positive quantity and the inequality clearly holds if, ${\nu_0}^2 < \frac{n_1^2 -(m_1^2 +2)}{l_1^2}$ holds. Since $\nu_0 \leq \nu_+$ this inequality is satisfied provided (\ref{eqn312}) holds. Thus, $\Psi(\nu_0) >0,$ and hence, no change of stability (instability) occurs.

Again, from (\ref{eqn311}), we observe that $\Psi(0) <0 $ if $n^2 +2n +n_1^2 < m_1^2$ holds and $\Psi(\nu) >0$ for large $\nu$ and therefore, $\Psi({\nu})=0$  has a real positive solution. Let $\nu_+$ denote smallest such root. Now (\ref{eqn311}) may be written as

\begin{equation}
\label{eqn313}
M\cos\theta+N\sin\theta=L,
\end{equation} 
where, $M=2(n-l\nu^{2})$, $N=2(m\nu-\nu^{3}),$$L=\nu^{6}+(l^{2}-2m)\nu^{4}+(m^{2}-2ln-l_{1}^{2})\nu^{2}+(n^{2}+n_{1}^{2}-m_{1}^{2})$ and $\theta = \tau + \delta.$ For the value of $\nu_+$ so obtained, we have from (\ref{eqn313}), $\theta \equiv \tau+\delta=\frac{1}{\nu_+} tan^{-1}(\frac{m{\nu_+} -{\nu_+}^{3}}{n-l{\nu_+}^{2}}).$ The proof is now complete. $\square$

\end{proof}

\section{Global Stability Results}
\label{sec4}
We consider an important special case of (\ref{eqn21}) in which $f(x,y)\equiv xy,$ the simple interaction term. We further assume that, $p(y)\equiv y$ and $V(x)\equiv x$, the recovery and the vaccination, are linear, then (\ref{eqn21}) takes the form
\begin{eqnarray}
\label{eqn41}
x' &=& a -b xy- d x -cx +\alpha z \nonumber\\
y' &=& b_1 x(t-{\tau})y- ry -d_1 y \nonumber\\
z' &=& r y(t-{\delta}) -\alpha z,
\end{eqnarray} 
for which the equilibrium points are given by the solutions  of
\begin{eqnarray}
\label{eqn42}
b x^{*}y^{*}+ d x^{*}+cx^{*}-\alpha z &=& a  \nonumber\\
b_1 x^{*}y^{*}- ry^{*} -d_1 y^{*} &=& 0 \nonumber\\
r_1 y^{*} -\alpha z^{*} &=& 0.
\end{eqnarray}
Clearly, $(\frac{a}{c+d},0,0)$ is always a solution of (\ref{eqn42}) and is a disease-free equilibrium of (\ref{eqn41}). If $y^{*}\neq 0$,then  from (\ref{eqn42}), $x^{*}=\frac{d_{1}+r}{b_{1}},$$y^{*}=\frac{a-(c+d)x^{*}}{bx^{*}-r} =\frac{b_1 a - (c+d)(d_1+r)}{bd_1 +r(b -b_1)}$ and $z^{*}=\frac{r}{\alpha}y^{*} = \frac{r}{\alpha} \frac{b_1 a - (c+d)(d_1+r)}{bd_1+ r (b-b_1)}$giving rise to a positive equilibrium solution of (\ref{eqn41}), provided that $bx^{*}\neq r,$ $(a-(c+d)x^{*})(bx^{*}-r)>0$ or simply if
\begin{equation}
\label{eqn43}
\frac{d_{1}+r}{b_{1}}< \frac{a}{c+d},
\end{equation}
holds. Since, $b_1 \leq b,$ $\frac{r}{b} < \frac{d_{1}+r}{b_{1}} =x^*,$ we have, $bx^*-r>0.$ \\ Hence, (\ref{eqn43}) is a necessary and sufficient condition for the existence of a positive (endemic) equilibrium for (\ref{eqn41}). In case $x^{*}=\frac{r}{b}=\frac{d_{1}+r}{b_{1}}$, entire $YZ$ plane becomes the equilibria. Now assuming $(x^{*},y^{*},z^{*})$ is a positive equilibrium of (\ref{eqn41}) we formulate the Theorem \ref{th41} as 

\begin{theorem}
 \label{th41}
The positive equilibrium solution $(x^{*},y^{*},z^{*})$ of (\ref{eqn41}) is globally asymptotically stable, independent of time delays, provided the parameters satisfy the condition $b_{1}<\min\{\frac{b(d_{1}+r)}{r},(c+d)\}$. 
\end{theorem}

\begin{proof}[Proof of Theorem \ref{th41}]
 Using (\ref{eqn42}) in (\ref{eqn41}) we rewrite (\ref{eqn41}) as 
\begin{eqnarray}
\label{eqn44}
x' &=& -by(x-x^{*})-bx^{*}(y-y^{*})-(c+d)(x-x^{*})+\alpha(z-z^{*})  \nonumber \\
y' &=& b_{1}(x(t-\tau)-x^{*})y \nonumber \\
z' &=& r(y(t-\delta)-y^{*})-\alpha(z-z^{*}).
\end{eqnarray}

Now consider the function $$ V(x,y,z)\equiv |x-x^{*}| +|\log y-\log y^{*}|+ |z-z^{*}|+$$\\$$ b_{1}\int^t_{t-\tau}|x(u)-x^{*}|du+r\int^t_{t-\delta}|y(u)-y^{*}|du.$$ Clearly $V(x^{*},y^{*},z^{*})=0$ and $V(x,y,z)\geq 0$ for $x\geq 0,\; y\geq 0,\; z\geq 0$ and $V\rightarrow\infty$ as $x,y,z\rightarrow\infty.$

The Upper Dini derivative of $V$, along the solutions of (\ref{eqn44}), is

\begin{eqnarray*}
D^{+}V &\leq & -(by+c+d)|x-x^{*}|-bx^{*}|y-y^{*}|+\alpha|z-z^{*}| \\
     && + b_{1}|x(t-\tau)=x^{*}|+r|y(t-\delta)-y^{*}|-\alpha(|z-z^{*}|)\\
     && +b_{1}|x-x^{*}|-b_{1}|x(t-\tau)-x^{*}| \\
     && +r|y-y^{*}|-r|y(t-\delta)-y^{*}|\\
     &\leq & -by|x-x^{*}|-(bx^{*}-r)|y-y^{*}|-(c+d-b_{1})|x-x^{*}|\\
     &\leq & -(c+d-b_{1})|x-x^{*}|-(bx^{*}-r)|y-y^{*}| \\
     &<0,& \quad \mbox{by hypotheses}.
 \end{eqnarray*}
The rest of the argument follows from standard arguments and the proof is complete. $\square$ \\
\end{proof}

We shall now provide a result for the global stability of the disease-free equilibrium. Suppose that $x(t) < \frac{d_1 +r}{b_1},$ for all $t.$ Now solving the second equation of (\ref{eqn41}) as a linear equation in $y,$ we get, $$y(t) = y(0) e^{(b_1 x_{\tau} - d_1 -r)t} \rightarrow 0$$ for sufficiently large $t.$ Using this in the third equation of (\ref{eqn41}), we have for large $t,$ $$z(t) \rightarrow z(0) e^{-\alpha t} \rightarrow 0,$$ then the first equation of (\ref{eqn41}) for sufficiently large $t$ becomes $$x'(t) = a -(c+d)x(t)$$ whose solution is $$x(t) = \frac{a}{c+d} + x(0) e^{-(c+d)t}.$$
 Thus, we have formulate the Theorem \ref{th42} as,
\begin{theorem}
\label{th42}
The disease free equilibrium $(x^{*},0,0)$ of (\ref{eqn41}) is globally asymptotically stable if the parameters satisfy the condition $\frac{aS}{c+d}<\frac{d_{1}+r}{b_{1}}.$ $\square$
\end{theorem}

\begin{remark}
Comparing the parametric conditions of Theorem \ref{th41} and Theorem \ref{th42}, we may notice that the existence of a positive equilibrium destabilizes the disease-free equilibrium and vice versa. $\square$ \\
\end{remark}

We shall now present some examples to illustrate the results as well as understand the influence of vaccination and treatment efforts.

\section{Examples and Simulations}
\label{sec5}
We have carried out simulations using subroutines developed in Matlab. Further, we have employed built-in function DDE23 \cite{33a}, a variable step size numerical integration routine for solving delay differential equations in Matlab. For more details on the DDE23 solver readers may refer to www.mathworks.com/dde\_tutorial. We have utilized the plot function and other advanced graphical tools available in Matlab to obtain results and figures presented in this work.

In all figures, curves drawn with dash-dot line and diamond marker ({$-.\meddiamond-.$}) denote exposed populations, while infected are represented by dashed line with pentagram marker ({$--\pentagram--$}) and curves denoted by solid line with square marker ({$-\medsquare-$}) stand for recovered populations. Simulations are carried out for various lengths of delays in all three cases $\tau >0,\, \delta =0,$ $\tau =0,\, \delta >0$ and $\tau >0,\, \delta >0.$ Figures generated stand for all these cases unless specified. Whenever there is a change in stability for a particular length of delay, the delay length is specified.

First, we consider the system (\ref{eqn41}) or (\ref{eqn21}) with $f(x,y)\equiv xy$, $V(x)\equiv x$ and $P(y)\equiv y$ to illustrate the results. 

\subsection{\textbf{Example 5.1}}
\label{ex51}
Consider the system,
\begin{eqnarray}
\label{eqn51}
x' &=& 10-xy-2x+z \nonumber \\
y' &=& x(t-\tau)y-2y \nonumber \\
z' &=& y(t-\delta)-z,
\end{eqnarray}
obtained by letting $b=c=d=d_{1}=b_{1}=\alpha=r=1$ and $a=10.$ Clearly (2,6,6) is a positive equilibrium of (\ref{eqn51}). It is easy to see that for $\tau = 0$, $\delta = 0$ (2,6,6) of (\ref{eqn51}) is stable by virtue of Theorem \ref{th31}.

For these values of parameters, we have $\overline{A}=756$, $\overline{B}=1488$,$\overline{C}=500$,$\overline{D}=84,$ all are positive and equation (\ref{eqn35}) becomes $756 T^{3}+1488 T^{2}+500T+84=0$ and has no solution for $T>0.$ Thus, the positive equilibrium of the system (\ref{eqn51}) remains stable for any length of delay $\tau>0$ when $\delta=0.$ 

Again the conditions of (\ref{eqn38}) are satisfied, and hence, by Theorem \ref{th35}, $\delta >0,\;\tau=0$ have no influence on the stability. Also, it may be noticed that the parametric conditions of Theorem \ref{th41} are satisfied here, and hence, $(2,6,6)$ is globally asymptotically stable by virtue of Theorem \ref{th41}, showing that the infection is widely prevalent. At the same time, conditions of Theorem \ref{th42} are violated, hence, no possibility of a disease-free environment. These dynamics may be observed in Figure \ref{fig3}.

\begin{figure}[!th]
  \includegraphics[height=60mm,width=\textwidth]{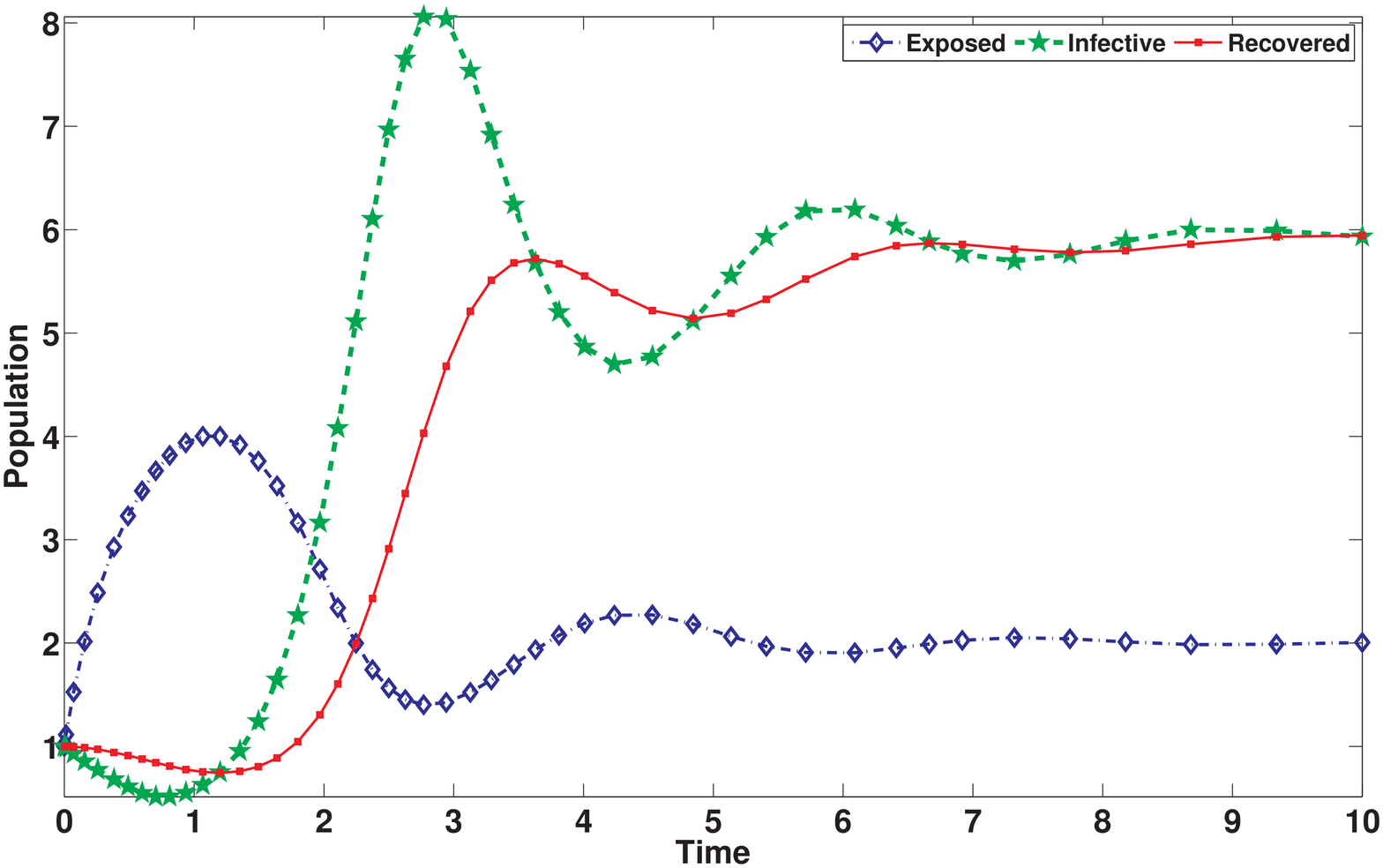}\\
  \caption{Global stability of endemic equilibrium $(2,6,6)$ of System (\ref{eqn51}) in Example 5.1 for all $\tau \geq 0,\; \delta \geq 0$.}\label{fig3}
\end{figure}

\subsection{\textbf{Example 5.2}}
\label{ex52}
Consider the system,
\begin{eqnarray}
\label{eqn52}
x' &=& 10-xy-2x+z \nonumber \\
y' &=& x(t-\tau)y-5y \nonumber \\
z' &=& 4y(t-\delta)-z,
\end{eqnarray}
wherein,  all the parameters are same as in (\ref{eqn51}) except that $r=4$ here. This example studies the influence of a higher treatment rate of the effected people than in earlier example. Clearly $x^{*}=\frac{d_{1}+r}{b_{1}} = 5,$$y^{*}=\frac{a-(c+d)x^{*}}{bx^{*}-r} = 0,$ $z^* = 0.$ Thus, $(5,0,0)$ is an equilibrium solution of (\ref{eqn52}). \\

The characteristic equation (\ref{eqn24}) further reduces to $F(\lambda)=\lambda^{3}+3\lambda^{2}+2\lambda =0$ for which the roots are $\lambda=0,\;\frac{-3\pm\sqrt{5}}{2}.$ This clearly shows that (\ref{eqn52}) is stable and behaves as a delay free system. This is further supported by the observation that the conditions of Theorem \ref{th42} for the global stability of $(5,0,0)$ are satisfied, and thus, the disease has no influence. Dynamics of the system (\ref{eqn52}) are shown in Figure \ref{fig4}. 

\begin{figure}[th]
  \includegraphics[height=60mm,width=\textwidth]{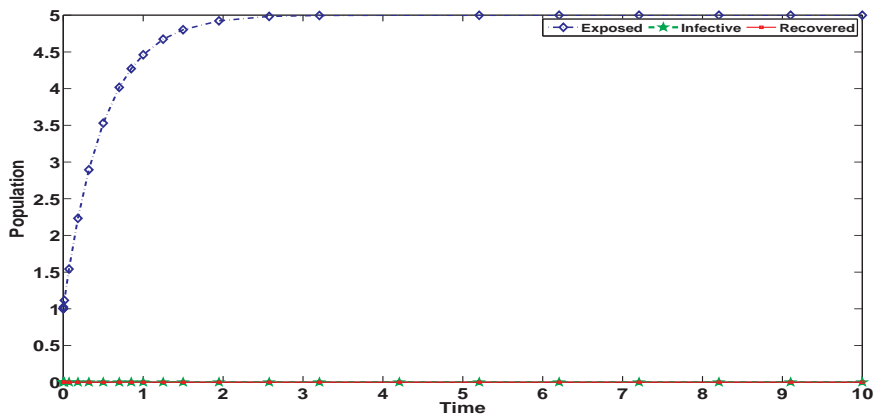}\\
  \caption{Disease free and delay free environment in System (\ref{eqn52}): solutions approaching disease free equilibrium $(5,0,0)$ for all lengths of delays under a high treatment rate.}
	\label{fig4}
\end{figure}

\subsection{\textbf{Example 5.3}}
\label{ex53}
Consider the system,
\begin{eqnarray}
\label{eqn53}
x' &=& 10-xy-4x+z \nonumber \\
y' &=& x(t-\tau)y-2y \nonumber \\
z' &=& y-z.
\end{eqnarray}

We now study the influence of choosing higher amount of vaccination i.e, we set the parameter $c=3$ in the Example \ref{ex51}. Clearly $x^{*}=\frac{d_{1}+r}{b_{1}}=2$, $y^{*}=\frac{a-(c+d)x^{*}}{bx^{*}-r}=2,$ $z^{*}=2.$ \\

Applying Theorem \ref{th31}, we get stability of the delay-free system. Now (\ref{eqn33}) takes the form $G(T)\equiv 42T^{3}-46T^{2}-117T-8=0.$ This equation has a positive solution for T, $T_{+}=2.325 (approx).$ Further, $\nu^{2}=\frac{3-T}{1+6T}>0$ yields $\nu=\nu_{+}=0.2125.$ For these values, we have $\tau_{+}=\frac{2}{\nu_{+}}[\tan^{-1}\nu T ] =4.3204.$ As $G'(T)>0$ for $T>2$ there is no further change of stability after $\tau=\tau_{+}=4.3204.$ Thus, (2,2,2) remains unstable for $\tau>\tau_{+}=4.3204.$ Since $a_0 =-4 <0,$ the instability of $(2,2,2)$ prevails for $\tau > \tau_{+}$ as per Theorem \ref{th33}. Again the condition $n^2+2n+n_{1}^2 < m_{1}^2 $ of Theorem \ref{th36} is satisfied, also indicating a change in stability of the delay free system in general case. The behaviour of solutions of (\ref{eqn53}) for various lengths of delay are shown in Figures \ref{fig5}, \ref{fig6} and \ref{fig7}.

\begin{figure}
	\begin{subfigure}[b]{0.5\textwidth}
                \includegraphics[height=50mm,width=\textwidth]{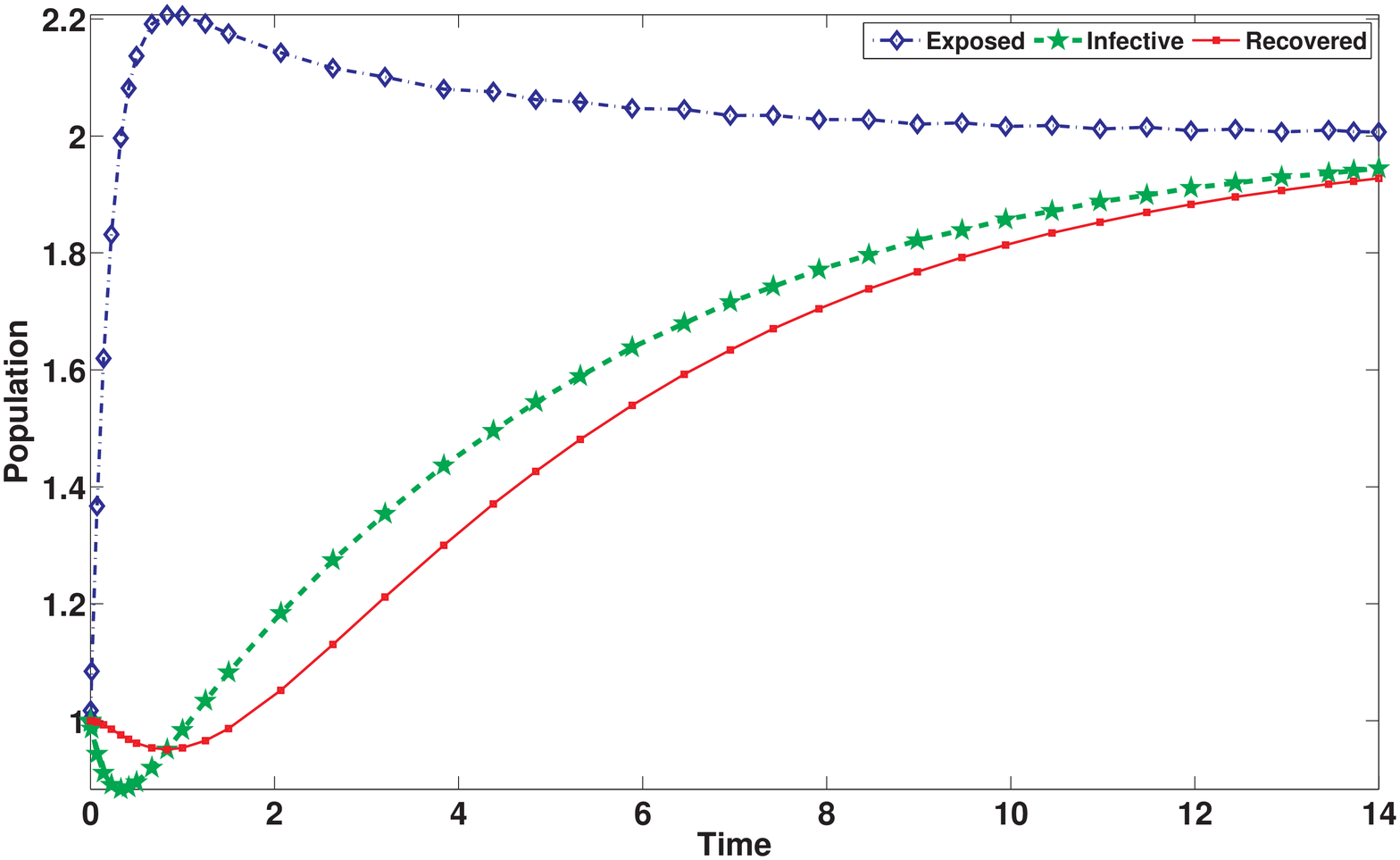}
                \caption{$\tau=0$}
                \label{fig5}
        \end{subfigure}%
        ~ 
        \begin{subfigure}[b]{0.5\textwidth}
                \includegraphics[height=50mm,width=\textwidth]{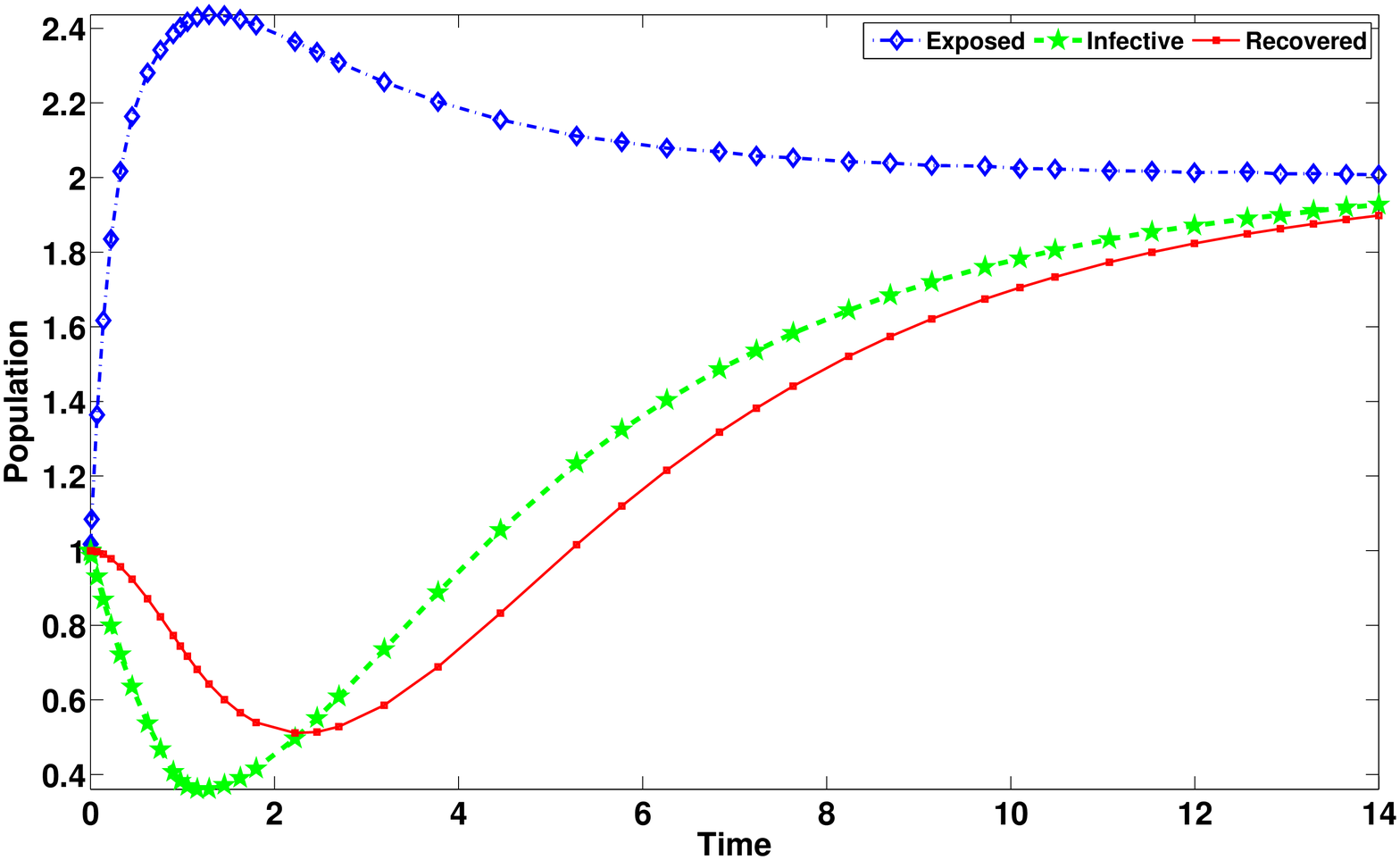}
                \caption{$\tau =0.9$}
                \label{fig6}
        \end{subfigure}
				\hfill
		\begin{subfigure}[b]{\textwidth}
                \includegraphics[height=50mm,width=\textwidth]{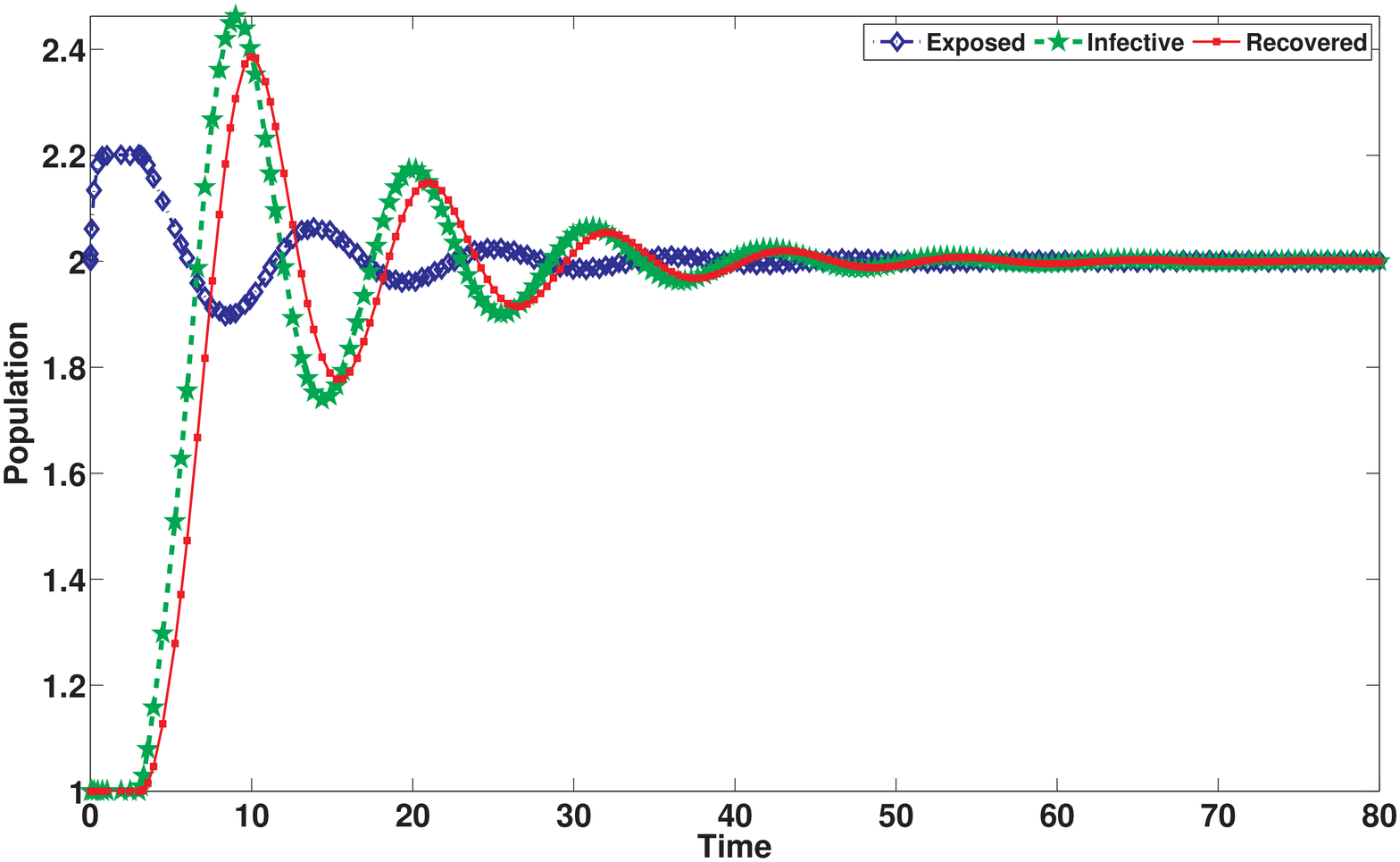}
                \caption{$\tau = 3$}
                \label{fig7}
        \end{subfigure}
  \caption{Solutions of (\ref{eqn53}) for various incubation delays; Endemic equilibrium (2,2,2) is stable for small delays under higher vaccination effort as compared to system (\ref{eqn51}).}
    \label{fig5to7}
	\end{figure}
	
For $\tau=4$ damped oscillations are observed as shown in Figure  \ref{fig8}, whereas for $\tau=5,6,7,8,9$, periodic solutions are observed in the simulations as may be seen from Figure \ref{fig9}-\ref{fig13} respectively. Thus, simulations support theoretical predictions.
\newpage
\setcounter{subfigure}{0}
\begin{figure}
	\begin{subfigure}[b]{0.5\textwidth}
                \includegraphics[height=50mm,width=\textwidth]{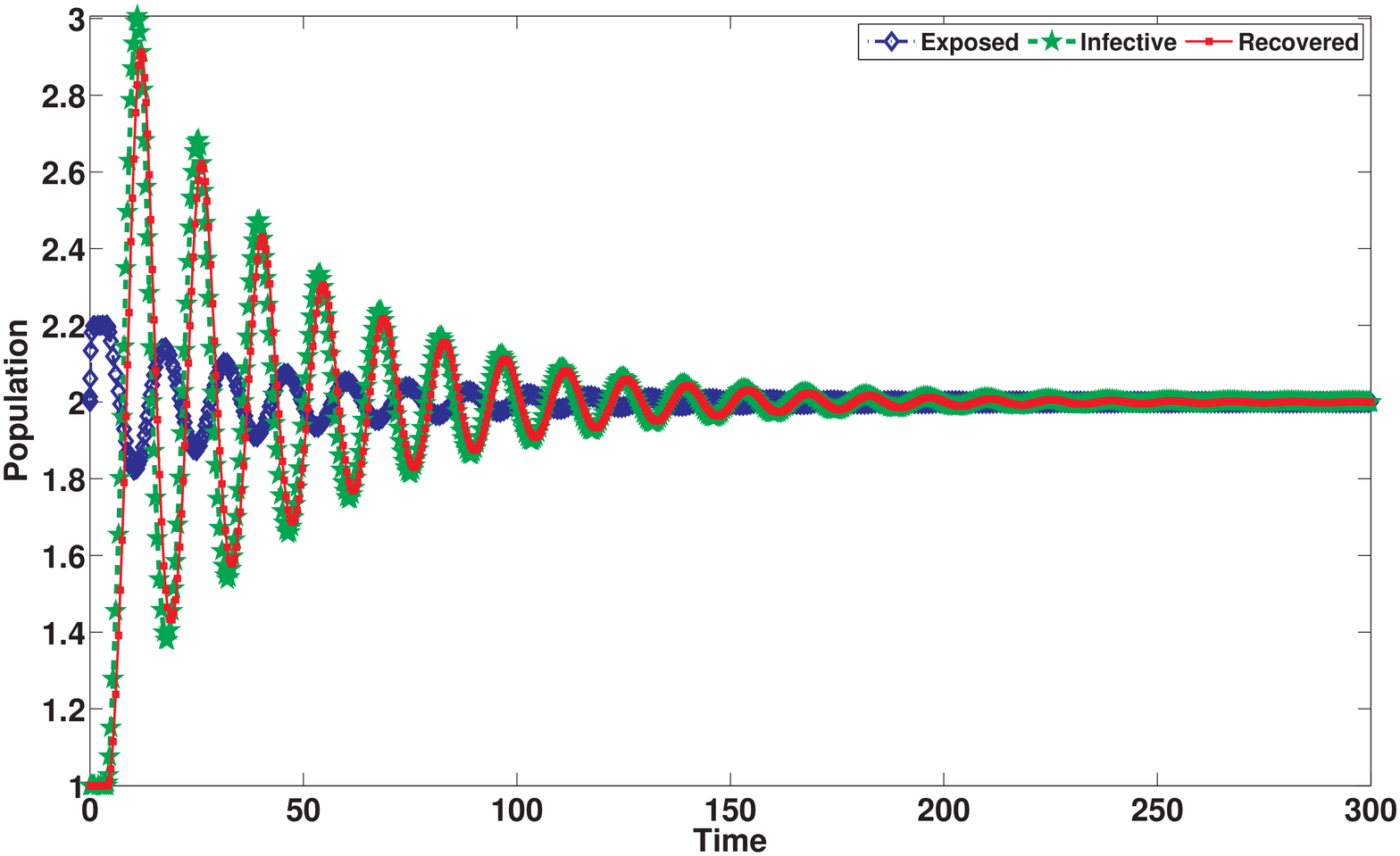}
                \caption{$\tau=4$}
                \label{fig8}
        \end{subfigure}%
        ~ 
        \begin{subfigure}[b]{0.5\textwidth}
                \includegraphics[height=50mm,width=\textwidth]{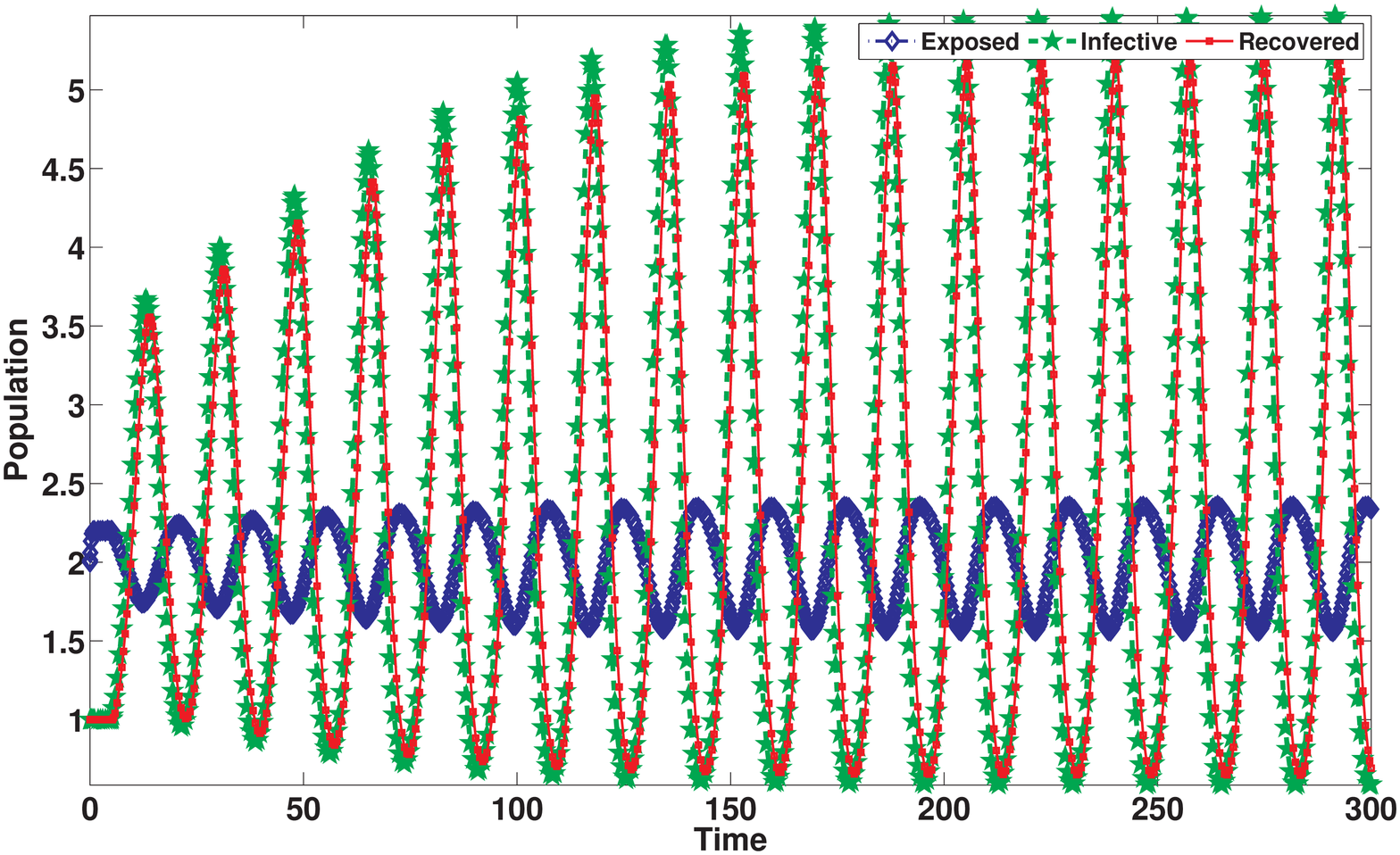}
                \caption{$\tau =5$}
                \label{fig9}
        \end{subfigure}
				\hfill
				\begin{subfigure}[b]{0.5\textwidth}
                \includegraphics[height=50mm,width=\textwidth]{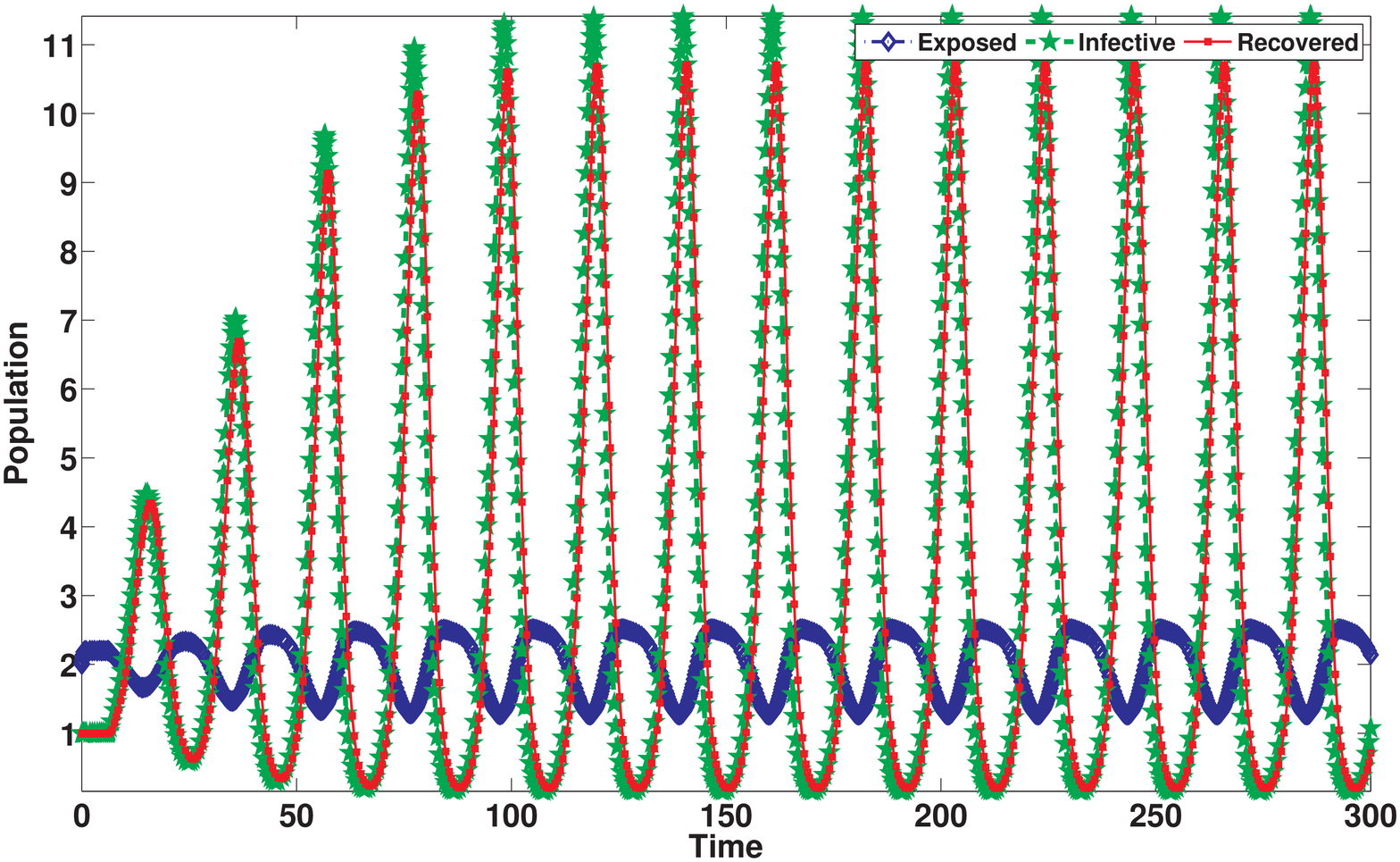}
                \caption{$\tau =6$}
                \label{fig10}
        \end{subfigure}
					\begin{subfigure}[b]{0.5\textwidth}
                \includegraphics[height=50mm,width=\textwidth]{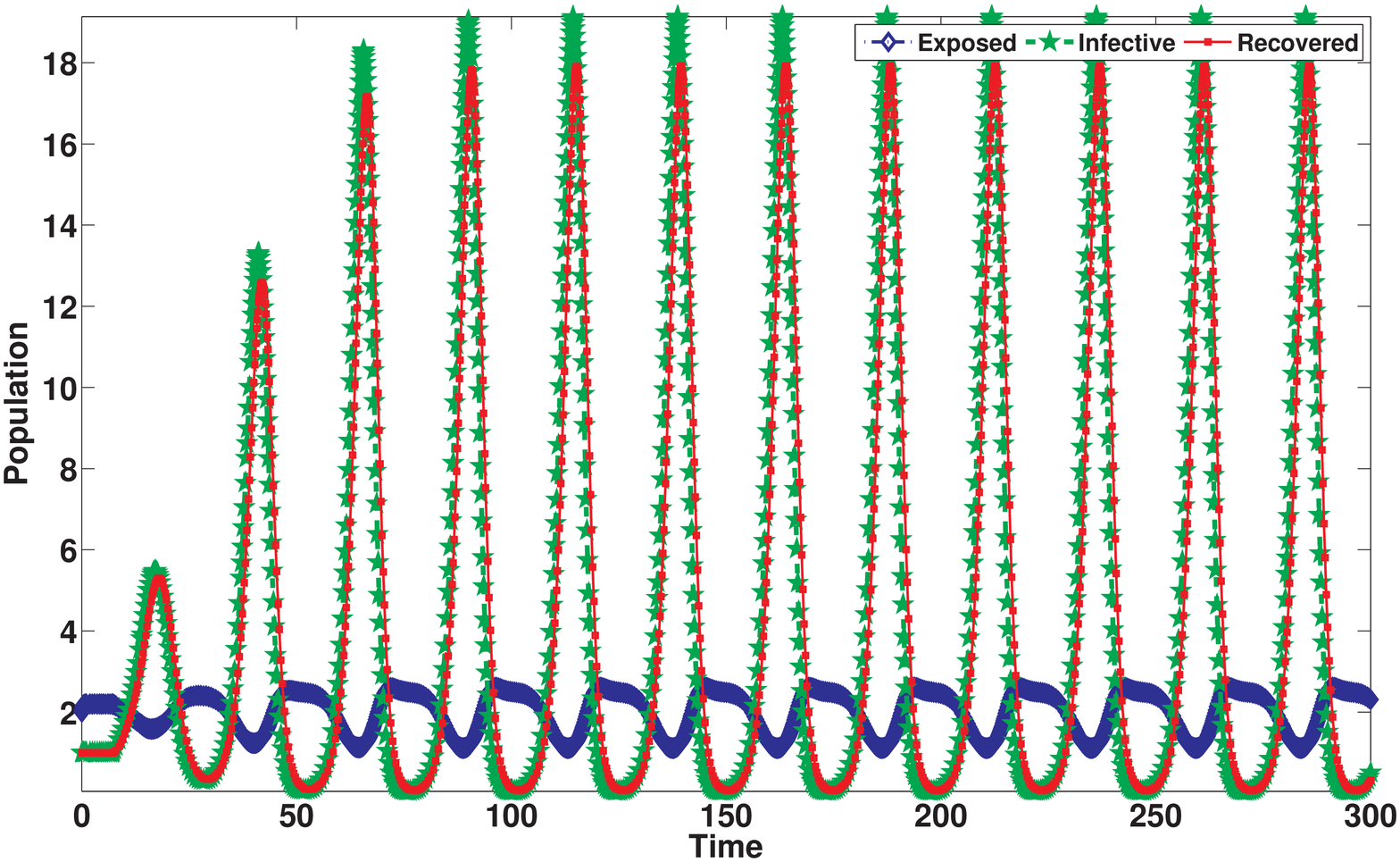}
                \caption{$\tau =7$}
                \label{fig11}
        \end{subfigure}
					\hfill
	\begin{subfigure}[b]{0.5\textwidth}
                \includegraphics[height=50mm,width=\textwidth]{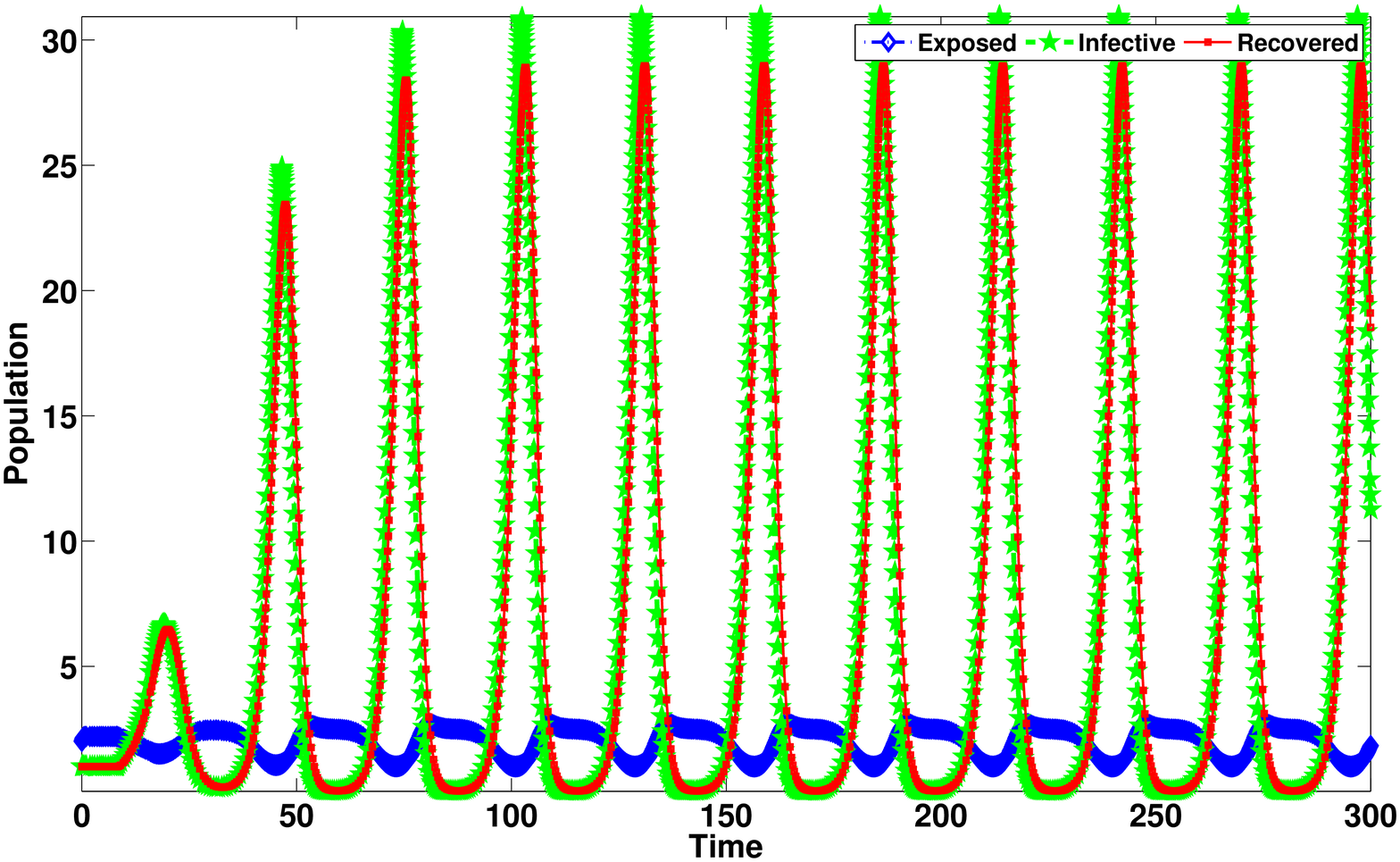}
                \caption{$\tau =8$}
                \label{fig12}
        \end{subfigure}
					\begin{subfigure}[b]{0.5\textwidth}
                \includegraphics[height=50mm,width=\textwidth]{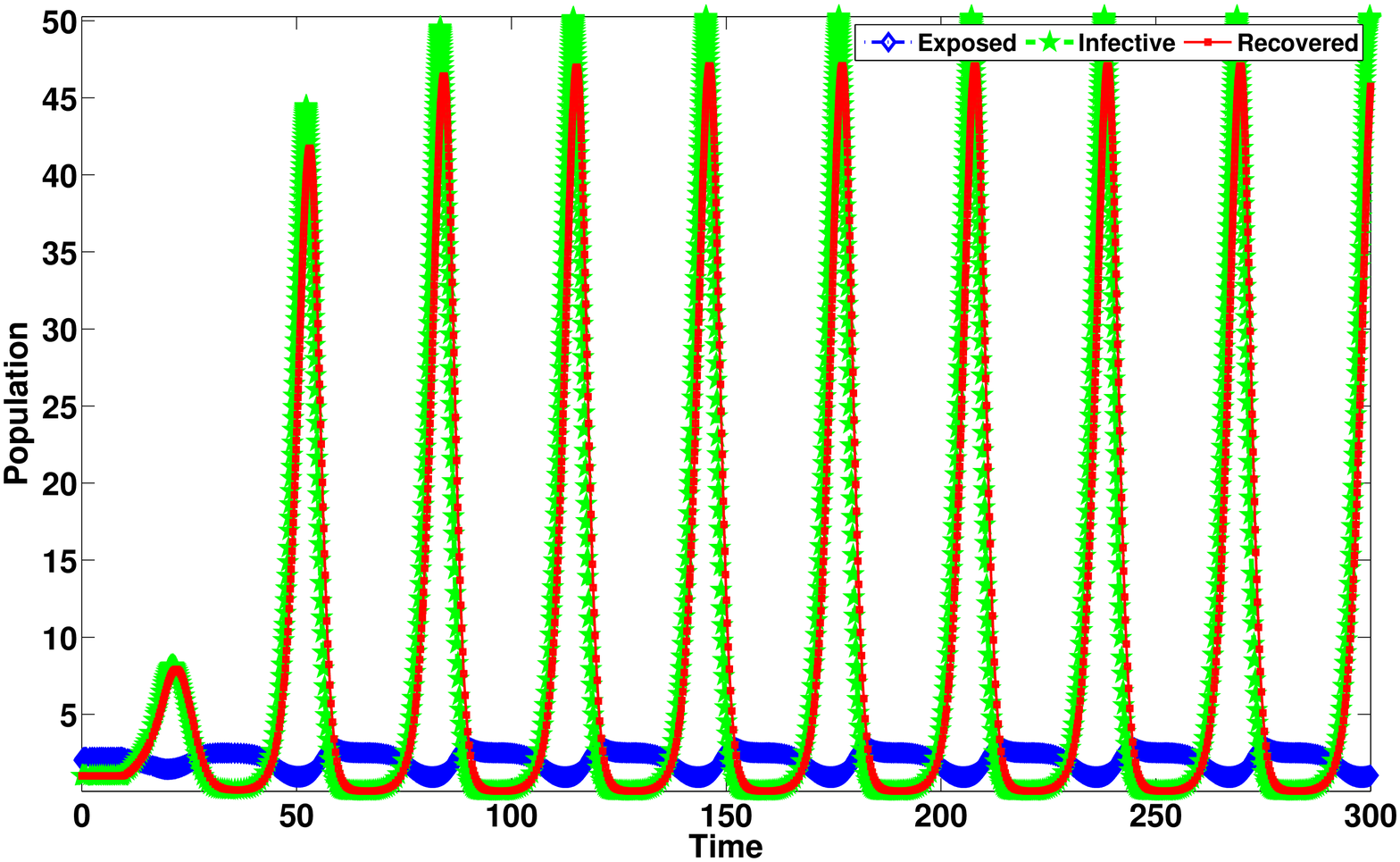}
                \caption{$\tau =9$}
                \label{fig13}
        \end{subfigure}				
  \caption{Solutions of (\ref{eqn53}) for various incubation delays; System (\ref{eqn53}) is behaving violently as delay$\tau$ goes on increasing beyond the estimated value of $\tau_{+} =4.3204$ showing that endemic equilibrium is unstable.}
    \label{fig8to13}
	\end{figure}

\newpage
\subsection{\textbf{Example 5.4.}}
\label{ex54}
Consider the system,
\begin{eqnarray}
\label{eqn54}
x' &=& 10-2xy-4x+z \nonumber \\
y' &=& 2x(t-\tau)y-5y \nonumber \\
z' &=& 4y(t-\delta)-z
\end{eqnarray}
We have chosen $r=4,\; c=3,\; b=2,\;b_1=2,\; d=d_1=1$ here. We shall study the influence of high recovery rate and vaccination rate together when the infection rate is doubled when compared to earlier examples. Clearly $x^{*}=\frac{aS}{c+d} = \frac{d_{1}+r}{b_{1}} = 2.5,$$y^{*}=\frac{a-(c+d)x^{*}}{bx^{*}-r} = 0,$ $z^* = 0.$ Thus, $(2.5,0,0)$ is an equilibrium solution of (\ref{eqn54}).

As in Example \ref{ex52}, the characteristic equation is $F(\lambda)=\lambda^{3}+3\lambda^{2}+4\lambda=0$ for which both the non zero roots certainly have negative real parts. Thus, the disease-free equilibrium is locally stable.

In the subsequent examples we choose non-linear functions for $f$,$V$ and/or $p.$ 

\subsection{\textbf{Example 5.5.}}
\label{ex55}
Consider the system,
\begin{eqnarray}
\label{eqn55}
x' &=& 10-2\frac{x}{x+y}-x+z \nonumber \\
y' &=& 2\frac{x(t-\tau)}{x(t-\tau)+y}-4y \nonumber \\
z' &=& 3y(t-\delta)-z
\end{eqnarray}
wherein, $r=3,\; c+d=1,\; b=2=b_1,\; d_1=1$ and $\alpha=1$ and non-linearity in infection only. Clearly this system has $(\frac{200}{21},\, \frac{10}{21},\,\frac{30}{21})$ as equilibrium which is locally stable by virtue of Theorem \ref{th35}. 

\subsection{\textbf{Example 5.6.}}
\label{ex56}
Consider the system,
\begin{eqnarray}
\label{eqn56}
x' &=& 10-3\frac{x}{x+2}y-x-x^2 +z \nonumber \\
y' &=& 3\frac{x(t-\tau)}{x(t-\tau)+2}y-2y \nonumber \\
z' &=& y(t-\delta)-z,
\end{eqnarray}
with non-linearities both in infection and vaccination. The parameters chosen are $r=3,\; c=1,\; d=1,\; b=3=b_1,\; d_1=1$ and $\alpha=1.$ Clearly this system has only a disease-free equilibrium given by $(\frac{\sqrt{41}-1}{2},\, 0,\,0)$ which is locally asymptotically stable for all delays $\tau$ and $\delta.$

\subsection{\textbf{Example 5.7.}}
\label{ex57}
We shall now introduce a non-linear recovery function in our model.
Consider the system,
\begin{eqnarray}
\label{eqn57}
x' &=& 1.5-3\frac{x}{x+2}y-x-\frac{x}{x+2} +z \nonumber \\
y' &=& 3\frac{x(t-\tau)}{x(t-\tau)+2}y-y-\frac{y}{1+y} \nonumber \\
z' &=& \frac{y(t-\delta)}{1+y(t-\delta)}-z,
\end{eqnarray}
wherein, $a=1.5$, $b=3=b_1,\; c=1,\; d=1=d_1=1$ $ r=1$ and $\alpha=1.$ Clearly this system has only a disease-free equilibrium given by $(\frac{\sqrt{57}-3}{4},\, 0,\,0)$ which is locally asymptotically stable for all delays $\tau$ and $\delta$ as may be seen in Figure \ref{fig14to17}.

\setcounter{subfigure}{0}
\begin{figure}
	\begin{subfigure}[b]{0.5\textwidth}
                \includegraphics[height=60mm,width=\textwidth]{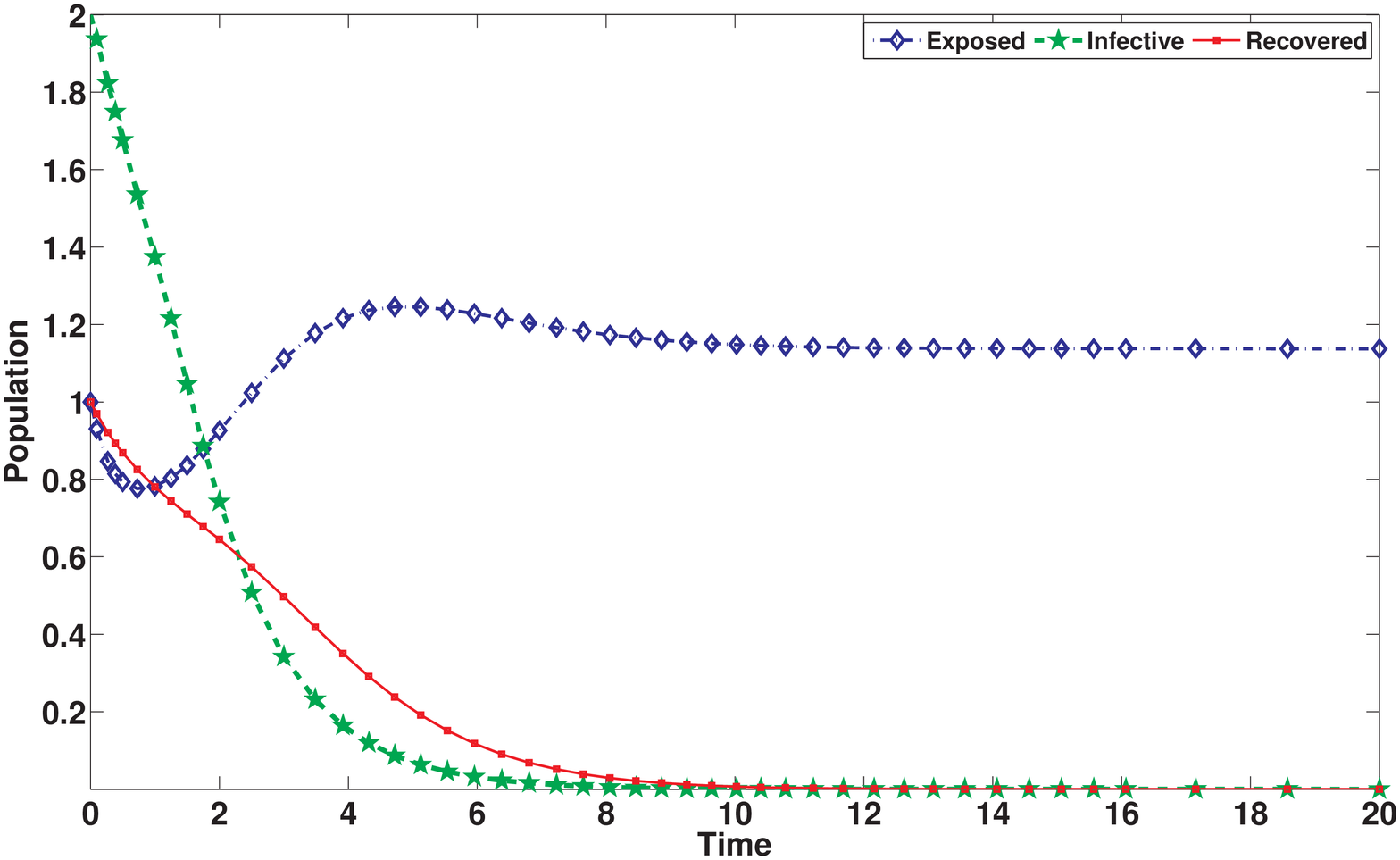}
                \caption{$\tau=1, \delta=0.5$}
                \label{fig14}
        \end{subfigure}%
        ~ 
        \begin{subfigure}[b]{0.5\textwidth}
                \includegraphics[height=60mm,width=\textwidth]{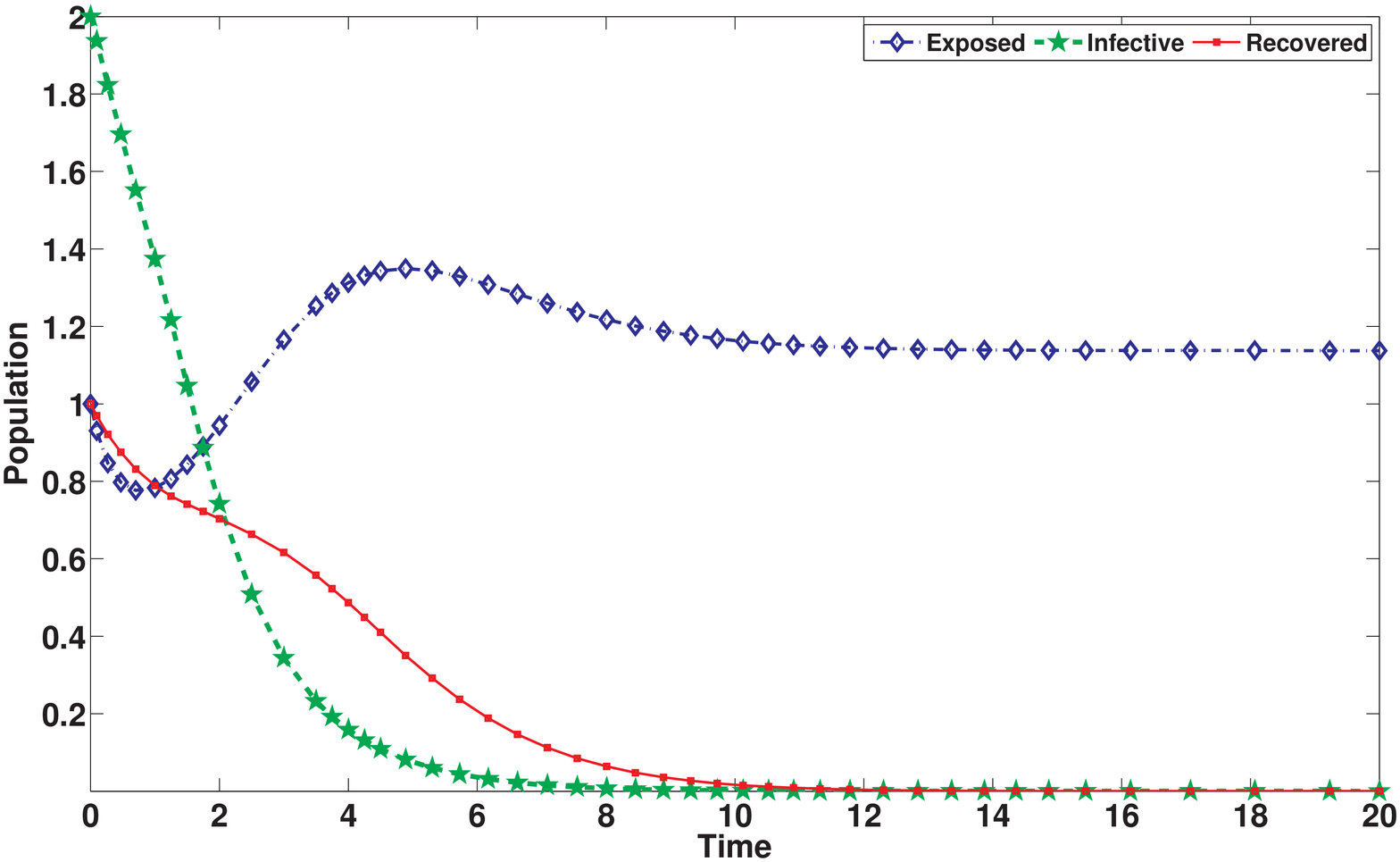}
                \caption{$\tau=1, \delta=1.5$}
                \label{fig15}
        \end{subfigure}
				\hfill
				\begin{subfigure}[b]{0.5\textwidth}
                \includegraphics[height=60mm,width=\textwidth]{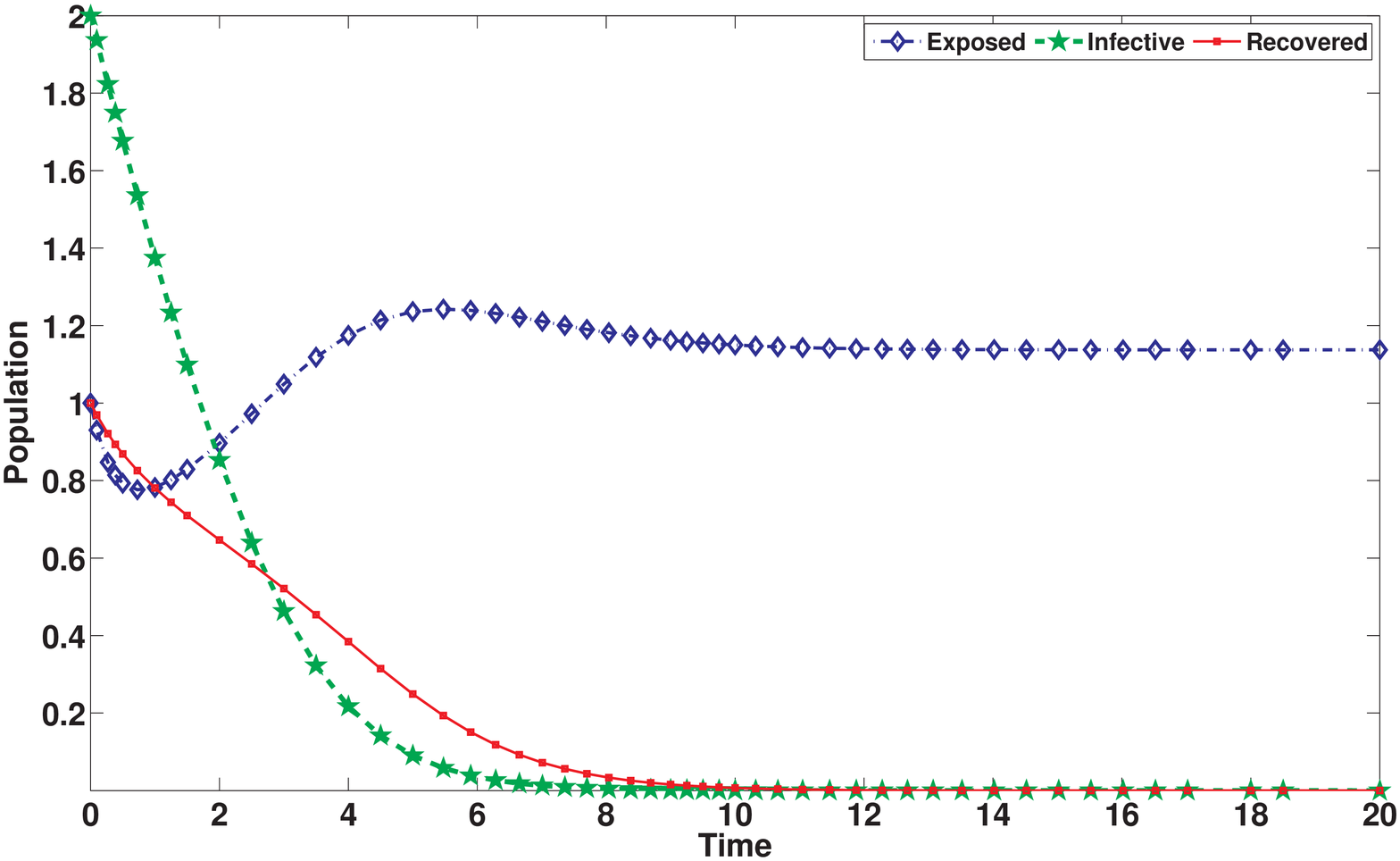}
                \caption{$\tau=9, \delta=0.5$}
                \label{fig16}
        \end{subfigure}
					\begin{subfigure}[b]{0.5\textwidth}
                \includegraphics[height=60mm,width=\textwidth]{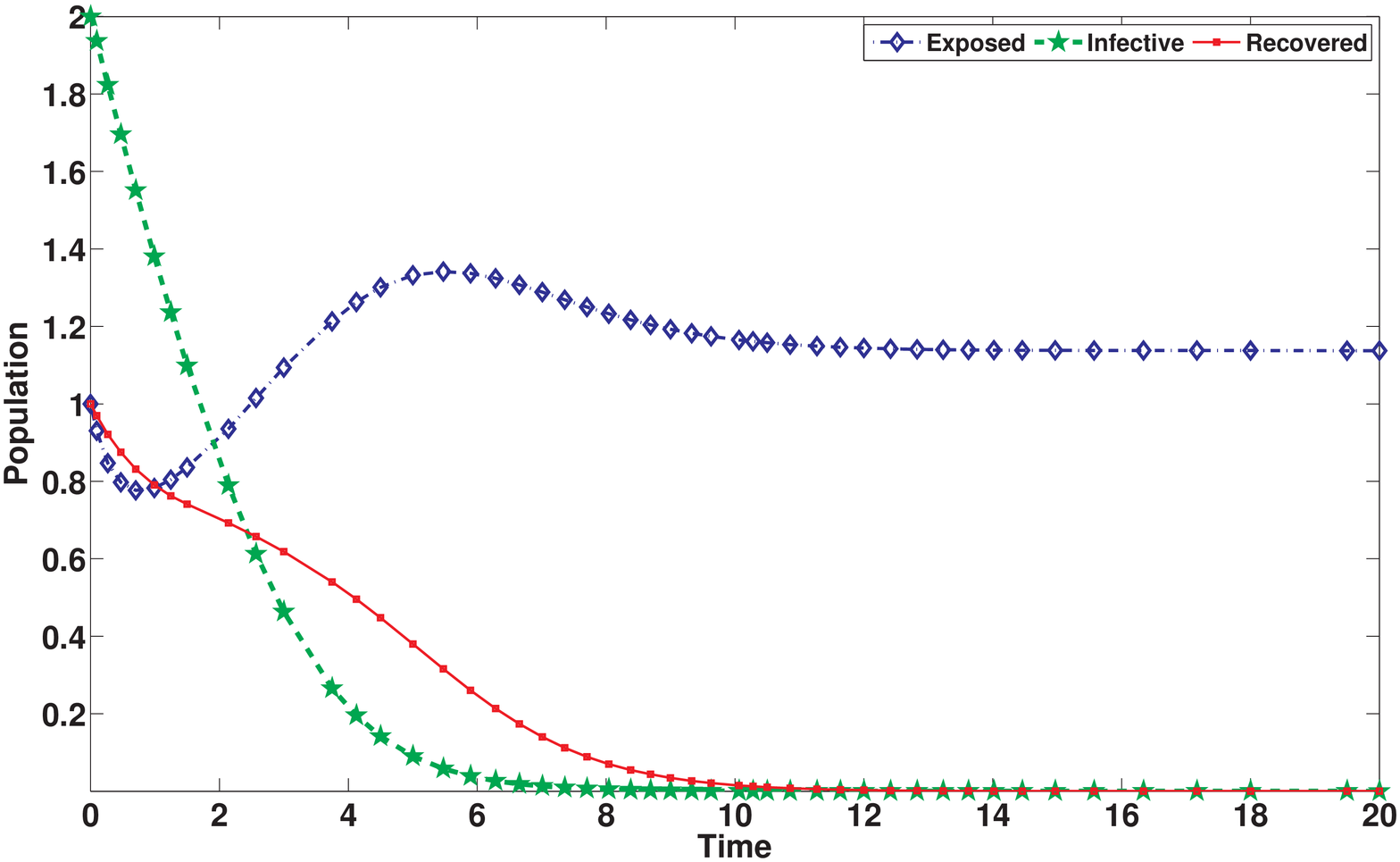}
                \caption{$\tau=9, \delta=1.5$}
                \label{fig17}
        \end{subfigure}
			
				\caption{Dynamics of the system (\ref{eqn57}) for various length of delays in both $\tau$ and $\delta$; Disease free equilibrium $(\frac{\sqrt{57}-3}{4},\, 0,\,0)$ remains stable under non-linearities in infection, vaccination and treatment.}
					\label{fig14to17}
\end{figure}

\subsection{\textbf{Observations}}

The System (\ref{eqn51}) in Example \ref{ex51} provides a situation where disease-free equilibrium $(5,0,0)$ and an endemic equilibrium $(2,6,6)$ do exist. However, conditions are conducive for both local and global stability of $(2,6,6).$ This indicates the prevalence of infection irrespective of a time delay and provide no scope for disease-free environment. \\

In Example \ref{ex52}, we have raised the treatment rate from $r=1$ in (\ref{eqn51}) to $r=4$ in (\ref{eqn52}). All other parameters and functional relations are kept as they are in (\ref{eqn51}). In this case only a disease-free equilibrium $(5,0,0)$ exists and there is no possibility of an endemic equilibrium. Theory says that $(5,0,0)$ is asymptotically stable both locally and globally creating an eventual infection-free environment. From this we notice that as more and more infected are treated, the possibility of infected contacting the susceptible comes down and spread of disease considerably reduced and more susceptible ($x^* =5$) remain in the system as compared to $x^*=2$ in the Example \ref{ex51}. \\

In order to understand the influence of vaccination effort on the system in Example 5.1, we raised the vaccination parameter from $c=1$ in system (\ref{eqn51}) to $c=3$ in system (\ref{eqn53}) (Example \ref{ex53}). All other parameters are kept as they are in Example \ref{ex51}. The system possesses two equilibria $(2.5,0,0)$ and $(2,2,2)$. When compared to the situation in Example \ref{ex51}, the equilibrium values of infected and recovered populations are considerably reduced from $y^*=6=z^*$ to $y^* =2=z^*,$ may be due to increased vaccination efforts. For small values of time delays, the disease equilibrium $(2,2,2)$ is stable but as delay in incubation increases the equilibrium becomes unstable as predicted by the theory. Wild fluctuations in populations are noticed for delay parameter $\tau >4$ which may be seen in Figures \ref{fig8}- \ref{fig13}. This indicates that time delay in incubation has an influence on the system even under enhanced vaccination. \\

In Example \ref{ex54} (system \ref{eqn54}), though the rate of conversion of susceptible into infected, that is $b_1$ has been doubled from $b_1 =1$ in (\ref{eqn51}) to $b_1=2$ in (\ref{eqn54}), we have stability of disease free equilibrium $(2.5,0,0)$ under the influence of higher vaccination and treatment $(c=3, r=4)$ in (\ref{eqn54}) compared to $c=1, r=1$ of (\ref{eqn51}).\\

The nonlinear function for infection chosen in Example \ref{ex55} (see system (\ref{eqn55})) provides no scope for the existence of a disease free equilibrium and infection prevails in all cases. Similarly, nonlinear functions of vaccination, infection and treatment are considered in Examples \ref{ex56} and \ref{ex57}. In both these cases, the corresponding disease free equilibrium is locally asymptotically stable independent of time delays.

\begin{remark}
For the systems in Examples \ref{ex52}, \ref{ex54}, \ref{ex56} and \ref{ex57}, we notice that only  a disease-free equilibrium exists. Further, $D=0,$ $C=0$ and $A\leq 0$ at this equilibrium and for the choice of functions made. As we have noticed in Remark \ref{th32}, the system is, thus, behaving like a delay-free system and the disease-free equilibrium is locally asymptotically stable, while Examples \ref{ex51}, \ref{ex53} and \ref{ex57} provide a disease environment. $\square$
\end{remark}
\section{Conclusions}
\label{sec6}
In this paper, we proposed a mathematical model describing the dynamics of a population consisting of susceptible, infected and recovered by treatment categories. Existence of equilibria and their local stability are studied in detail. An attempt is made to understand the influence of vaccination and treatment on the populations in the presence of time delays in infection, recovery. Results on global stability are also established for a special case of the proposed model. Situations where the effect of time delays on the stability of the equilibria is null and void are identified in this special case. It is important and interesting to further explore this view point for more general cases. In \cite{2} it is observed that the vaccination threshold rate is admissible, beyond which the infection gets eradicated. But Examples \ref{ex51} and \ref{ex53} illustrate that vaccination alone cannot eradicate infection but it is the treatment rate that is playing a key role in controlling the disease when the system is influenced by time delays.\\

The present study raises some basic questions such as,

\begin{enumerate}

\item In Example \ref{ex53}, it is noticed that the equilibrium $(2,2,2)$ is locally stable as long as the time delay in infection $\tau <4.$ Thereafter this equilibrium becomes unstable as predicted by the theory (see results in Sections \ref{sec3} and \ref{sec4}). At the same time, simulations show that for $\tau > 4$ the system is behaving very wild and large fluctuations are observed in infected and recovered populations. This situation is more vulnerable than having a stable endemic equilibrium where the situation may not be much beyond control. Thus, existence of a stable endemic equilibrium may sometimes be a better choice. Further, the behaviour of the system under the simultaneous influence of time delays (general case, Section \ref{subsec34}) needs to be investigated in detail. For commensurate delays, study in Section \ref{subsec34} may provide an estimate but for other types of delays, more explicit estimates are required. However, we need to analyze such situations further. \\

\item Most of the mathematical studies on biological models concentrate on the stability of equilibria (known or identifiable solutions) of the system. Consider the following system.
\begin{eqnarray}
x' &=& 10-3xy-2x+z \nonumber \\
y' &=& 3x(t-\tau)y-5y \nonumber \\
z' &=& 2y(t-\delta)-z.
\end{eqnarray}

This has both disease free $(5,0,0)$ and endemic $(\frac{5}{3}, \frac{20}{9}, \frac{20}{9})$ equilibria. But neither Theorem \ref{th42} nor Theorem \ref{th41} holds for this case. Also system (\ref{eqn55}) in Example \ref{ex55} allows no possibility for existence of a disease free equilibrium, and hence, no scope for disease free environment. Comments made in Section \ref{subsec21} are aimed at such situations. In such cases, we have to study the ways of controlling the infected populations rather than the stability of the system. This raises the following questions. How to control $y$? Can we find a small $\epsilon \geq 0$ such that $y(t)\leq \epsilon$ for some $t\geq T.$ Can we create an inactive state (a 'zone of no activation') for the spreading of the disease? That means can we identify and keep the infected people in a quarantine (at least till vaccination is effective or treatment available is sufficient to meet the requirements)? This would be the point of contention in our future exposition. \\

\end{enumerate}

It is observed in \cite{35} that for diseases like Chagas, the effort of vaccination is not adequate to control the disease and treatment is recommended in both acute and chronic stages of diseases. System (\ref{eqn21}) through Examples \ref{ex51} to \ref{ex57} illustrates such cases. Our examples conclude that when the treatment rate is high, a disease-free environment is created (stability of $(2,6,6)$ - Example \ref{ex51} to stability of $(5,0,0)$ - Example \ref{ex52}). In the presence of improved vaccination effort alone, the infected and hence, recovered populations are lowered (stability of $(2,6,6)$ - Example \ref{ex51} to stability of $(2,2,2)$ - Example \ref{ex53}). Under the influence of better vaccination in addition to high treatment rate, we notice a considerable decrease in susceptible population due to increased rate of infection (stability of $(5,0,0)$ - Example \ref{ex52} to stability of $(2.5,0,0)$ - Example \ref{ex54}). It would be interesting to see how far our model explains the dynamics of Chagas disease. \\

On the whole, our common expectations on the model really come true and this prompts us to analyze this model further with regard to its applicability for understanding various infectious diseases - this would be our scope of future study. We  need to present simpler conditions on parameters and provide results for global asymptotic stability for general infection functions $f(x,y).$ \\

Also, global stability analysis of the model with general interaction, vaccination and recovery functions is in the offing. Infection rate, treatment rate and vaccination are the three important parameters that  essentially decide the dynamics of the spread of a disease. An effort has to be made for estimation of these parameters in order to decide the measures of control of disease. Our investigation on these lines will be made open soon. \\

\section{Discussion and Open Research Problems}
\label{sec7}
We shall now propose some modifications to (\ref{eqn1}) taking into consideration some realistic phenomena. We present the following cases as open research problems that would form a part of our further research,

\begin{enumerate}

\item We have assumed the growth rate of susceptible population $a$ in (\ref{eqn1}) as a fixed constant. This assumption holds good when the susceptible population receives input from a large society and the disease visits for a limited period during which the reproduction of susceptible, infected and/or recovered population is not significant. However, in case of chronic diseases such as AIDS, Chagas e.t.c., offspring of all three populations may make sizeable contribution to the system. Thus, we may consider,

\begin{eqnarray}
x' &=& aU(x,y,z) -b f(x,y)- d x - cV(x)+\alpha z \nonumber\\
y' &=& b_1 f\big(x(t-\tau), y(t)\big)- r P(y) -d_1 y \nonumber\\
z' &=& r P\big(y(t-\delta)\big) -\alpha z,
\end{eqnarray}

here $aU(x,y,z)$ denotes the growth function of susceptible population.

\item Vaccination well before an attack by the disease is rare. May be it is possible in a well prepared society ? But in most of the cases, vaccine is not readily available and a delay in its preparation or supply is natural. This may be viewed in two directions.
\begin{enumerate}
\item Vaccine takes time $\gamma >0$ to develop immunity so that susceptible population vaccinated at time $t-\gamma$ are immune at time $t.$ Hence, such population may be removed from the system. Thus, we have 

\begin{eqnarray}
x' &=& a -b f(x,y)- d x - cV(t-\gamma)+\alpha z \nonumber\\
y' &=& b_1 f\big(x(t-\tau), y(t)\big)- r P(y) -d_1 y \nonumber\\
z' &=& r P\big(y(t-\delta)\big) -\alpha z.
\end{eqnarray}
Notice that we have taken $V(x(t)) \equiv V(t-\gamma)$ in (\ref{eqn1}) for this case.

\item Vaccination effort being carried out at time $t$ depends on a previous estimation of population. That is, we choose $V(x(t))= V(x(t-\eta))$ in (\ref{eqn1}). We have 

\begin{eqnarray}
x' &=& a -b f(x,y)- d x - cV(x(t-\eta))+\alpha z \nonumber\\
y' &=& b_1 f\big(x(t-\tau), y(t)\big)- r P(y) -d_1 y \nonumber\\
z' &=& r P\big(y(t-\delta)\big) -\alpha z.
\end{eqnarray}
\end{enumerate}

\item Recovered population may take time to become susceptible again. That means, they could retain health for some time, say, $\mu >0.$ This may be represented in (\ref{eqn1}) as,
 
\begin{eqnarray}
x' &=& a -b f(x,y)- d x - cV(x)+\alpha z(t-\mu) \nonumber\\
y' &=& b_1 f\big(x(t-\tau), y(t)\big)- r P(y) -d_1 y \nonumber\\
z' &=& r P\big(y(t-\delta)\big) -\alpha z.
\end{eqnarray}

\item Though mathematical models can describe many plausible realistic phenomena, their usefulness is acceptable only if they stand the test on the real time data. Thus, it would be interesting to test the results of this article with available real time data so that necessary modifications or improvement of (\ref{eqn1}) may be taken up suitably.

\item Consider the influence of sudden environmental changes or other measures taken up to control the spread of disease (may be described as noise in the system). Introducing such term in (\ref{eqn1}), we have

\begin{eqnarray}
x' &=& a -b f(x,y)- d x - cV(x)+\alpha z \nonumber\\
y' &=& b_1 f\big(x(t-\tau), y(t)\big)- r P(y) -d_1 y -\beta(t,y)\nonumber\\
z' &=& r P\big(y(t-\delta)\big) -\alpha z,
\end{eqnarray}
here $\beta(t,y) \geq 0$ for all $t,\; y$ may be understood as the rate at which infected population is decreasing owing to changes in environment that are non-conducive for infection or measures taken up to keep the infected away from the system (isolation) or measures taken up by society for controlling the spread of the disease. Of course, a $\beta(t,y)<0$ condition says that the situation is encouraging for growth of infected population.   
\end{enumerate}

\section*{Acknowledgements}
The authors wish to thank the anonymous reviewers for their useful and constructive suggestions.


\section*{References}
\bibliographystyle{elsarticle-num}

\begin{thebibliography}{00}

\bibitem{1}
Afia Naheed, Manmohan Singh, David Lucy. Numerical study of SARS epidemic model with the inclusion of diffusion in the system. Applied Mathematics and Computation (2014);2209:480--498.

\bibitem{2}
Almut Scherer, Angela Mclean. Mathematical models of vaccination. British Medical Bulletin (2002);62:187--199.

\bibitem{3}
Beretta E, Kuang Y. Modeling and analysis of a marine bacteriophage infection. Mathematical Biosciences (1998);149:57--76.

\bibitem{4}
Billarda L, Dayananda PWA. A multi-stage compartmental model for HIV-infected individuals: I – Waiting time approach. Mathematical Biosciences (2014);249:92--101.

\bibitem{5}
C.P. Bhunu. Mathematical analysis of a three-strain tuberculosis transmission model. Applied Mathematical Modelling (2011);35(9):4647--4660.

\bibitem{6} 
Chichigina O, Valenti D, Spagnalo B, A simple noise model with memory for biological systems, Fluctuation and Noise Letters (2005);5(2):L243--L250.

\bibitem{7}
Diehmann O, Heesterbeek J. Mathematical Epidemiology of Infectious Disease: Model building, analysis and interpretation. New York: Wiley; 2000.

\bibitem{8}
Erbe LH, H.I. Freedman HI , Sree Hari Rao V. Three species food-chain models with mutual interference and time delays. Mathematical Biosciences (1986);80:57--80.

\bibitem{9}
Fiasconaro A, Valenti D, Spagnolo B. Nonmonotonic behavior of spatiotemporal pattern formation in a noisy Lotka-Volterra system. Acta Phys. Pol. B. (2004);35:491--1500.

\bibitem{10}
Freedman HI, Sree Hari Rao V. Stability criteria for a system involving two time delays. SIAM J. Appl. Math (1986);46(4):552--560.

\bibitem{11}
Freedman HI, Sree Hari Rao V, Jaya Lakshmi K. Stability, persistence and extinction in predator - prey system with discrete and continuous time delays. WSSIAA (1992);1:221-238. 

\bibitem{12}
Farinaz Forouzannia, Abba B. Gumel. Mathematical analysis of an age-structured model for malaria transmission dynamics. Mathematical Biosciences (2014);247:80--94.

\bibitem{13}
Gang H, Takeuchi Yasuhiro, Ma Wanbiao, Wei Deijun. Global stability for SIR and SEIR epidemic models with  nonlinear incidence rate. Bulletin of Mathematical Biology (2010);72(5):1192--1207.

\bibitem{14} 
Gielen JLW. A stochastic model for epidemics based on renewal equation. J. Biological Systems (2000);8:1--20.

\bibitem{14a}
Gray A, Greenhalgh D, Hu L, Mao X, Pan J. A Stochastic Differential Equation SIS Epidemic Model. SIAM Journal on Applied Mathematics (2011);71(3):876--902.

\bibitem{15}
Guihua Li, Wang Wendi, Jin Chen. Global stability of a SIR epidemic model with constant immigration. Chaos, Solitons and Fractals (2006);30:1012--1019.

\bibitem{16}
Hall IM, Iain Barrass, Steve Leach, Didier Pittet, Stéphane Hugonnet. Transmission dynamics of methicillin-resistant Staphylococcus aureus in a medical intensive care unit. J. R. Soc. Interface (2012);9 (75):2639--2652.

\bibitem{16a}
Hampton T. Largest-Ever Outbreak of Ebola Virus Disease Thrusts Experimental Therapies, Vaccines Into Spotlight. JAMA (2014);312(10):987-989. doi:10.1001/jama.2014.11170.

\bibitem{17}
Han L, Pugliese A.  Epidemics in two competing species. Nonlinear Analysis: Real World Applications (2009);10(2):723--744.

\bibitem{18} 
Hertz D, Jury EI, Zeheb E. Simplified analytical stability test for systems with commensurate time delays. IEE Proceedings on Control Theory and Applications (1984);131(1):52--56.

\bibitem{19}
Hethcote HW, Lewis MA and Driessche PV.  An epidemic model with delay and a nonlinear incidence rate. Journal of Mathematical Biology (1989);27:49--64.

\bibitem{20}
Hethcote HW, Driessche PV.  Some epidemic models with nonlinear incidence. Journal of Mathematical Biology (1991);29:271--287.

\bibitem{21}
Hethcote HW, Stetch HW, Driessche PV. Nonlinear oscillations in epidemic models. SIAM J. Appl. Math. (1981);40:1--9.

\bibitem{21a}
Jihad Adnani, Khalid Hattaf, Noura Yousfi. Stability Analysis of a Stochastic SIR Epidemic Model with Specific Nonlinear Incidence Rate. International Journal of Stochastic Analysis (2013); 2013. doi:10.1155/2013/431257.

\bibitem{22}
JiangZ, Wei J. Stability and bifurcation analysis in a Delayed SIR model. Chaos, Solitons and Fractals (2008);35(3):609--619.

\bibitem{23}
Julie Nadeau, Connell McCluskey. Global stability for an epidemic model with applications to feline infectious peritonitis and tuberculosis. Applied Mathematics and Computation (2014);230:473-–483.

\bibitem{23a}
Keng-Cheng Ang. A simple stochastic model for an epidemic: numerical experiments with Matlab. The Electronic Journal of Mathematics and Technology (2007);1(2):117--128.

\bibitem{24}
Korobeinikov A, Maini PK.  A Lyapunov function and global properties for SIR and SEIR epidemiological models with nonlinear incidence. Mathematical Bioscience and Engineering (2004);1:52--60.

\bibitem{25}
Kyrychko YN, Blyuss KB. Global properties of a delayed SIR model with temporary immunity and nonlinear incidence rate. Nonlinear Analysis and Real World Applications (2005);6(3):495--507.

\bibitem{26} 
La Cognata A, Valenti D, Dubkov AA, Spagnolo B. Dynamics of two competing species in the presence of Levy noise sources. Phys. Rev. E. (2010);82:1--9.

\bibitem{27}
Leung IKC, Gopalsamy K.  Dynamics of continuous and discrete time siv models of Gonorrhea transmission. Dynamics of Continuous, Discrete and Impulsive Systems Series B: Applications \& Algorithms (2012);19:351--375.

\bibitem{28}
Liu W, Hethcote HW, Levin SA. Dynamical behaviour of epidemiological models with nonlinear incidence rates. Journal of Mathematical Biology (1987);25(4):359--380.

\bibitem{29}
Milner FA, Zhao R. A New Mathematical Model of Syphilis. Mathematical Modelling of Natural Phenomena (2010);5 (06):96--108.

\bibitem{30}
Naresh R, Tripathi A, Tchuenche JM, Dileep Sharma.  Stability analysis of a time delayed epidemic model with nonlinear incidence rate. Comput. Math. Appl. (2009);58(2):348--359.

\bibitem{31}
Puntani Pongsumpun, I-Ming Tang.  Dynamics of a New Strain of the H1N1 Influenza A Virus Incorporating the Effects of Repetitive Contacts. Computational and Mathematical Methods in Medicine (2014); 2014.


\bibitem{32}
Prakash S. Bisen, Ruchika Raghuvanshi. Emerging Epidemics: Management and Control. New York: Wiley; 2013.

\bibitem{33}
Ranjit Kumar Upadhyay, Nitu Kumari, Sree Hari Rao V. Modeling the spread of bird flu and predicting outbreak diversity. Nonlinear Analysis: Real World Applications (2008);9(4):1638 --1648.

\bibitem{33a}
Shampine LF, Thompson S. Solving DDEs in MATLAB, Applied Numerical Mathematics (2001);37:441-- 458.

\bibitem{34}
Sree Hari Rao V, Naresh Kumar M.  Estimation of the parameters of an infectious disease model using neural networks. Nonlinear Analysis: Real World Applications (2010);11(3):1810--1818.

\bibitem{35}
Sree Hari Rao V, Naresh Kumar M. Predictive Dynamics: Modeling for Virological Surveillance and Clinical Management of Dengue. In \textit{Dynamics Models of Infectious Diseases, Volume I: Vector Borne Diseases} (eds. V.Sree Hari Rao and Ravi Durvasula). pp 1--41 (New York: Springer; 2013). doi: 10.1007/978-1-4614-3961-5\_1.

\bibitem{35a}
Sree Hari Rao V, Naresh Kumar M. A New Intelligence-Based Approach for Computer-Aided Diagnosis of Dengue Fever. Information Technology in Biomedicine, IEEE Transactions on (2012);16(1):112--118. doi: 10.1109/TITB.2011.2171978.

\bibitem{35b}
Sree Hari Rao V, Naresh Kumar M. Control of Infectious Diseases: Dynamics and Informatics
. In \textit{Dynamics Models of Infectious Diseases, Volume II: Non Vector-Borne Diseases} (eds. V.Sree Hari Rao and Ravi Durvasula). pp 1--30 (New York: Springer; 2013). doi: 10.1007/978-1-4614-9224-5\_1.

\bibitem{36}
Sree Hari Rao V, Phaneendra BhRM. Stability of differential systems involving time lags. Advances in Mathematical Sciences and Applications (1998);8(2):948--964.

\bibitem{37}
Tchuenche JM, Nwagwo A. Local stability of a SIR model and effect of time delay. Mathematical Models in Applied Science (2009);32(16):2160--2175.

\bibitem{38}
Tchuenche JM, Nwagwo A, Levins R. Global behaviour of a SIR epidemic model with time delay. Mathematical models in Applied Science (2007);30(6):733--749.

\bibitem{39}
Tian X, Xu R.  Stability analysis of a delayed SIR epidemic model with stage structure and nonlinear analysis. Discrete Dynamics in Nature and Society (2009);17.

\bibitem{40}
Tuckwell HK, Le Corfec E. A stochastic model for early HIV-I population dynamics. J. Theoretical Biology (1998);195:451--463.

\bibitem{41} 
Valenti D, Tranchina L, Brai M, Caruso A, Cosentino D, Spagnolo B. Environmental metal pollution considered as noise: Effects on the spatial distribution of benthic forminifera in two coastal marine areas of Sicily(Southern Italy).   Ecological Modelling (2008);213:449-462.

\bibitem{42}
Van Den Driessche P, Wang L, Zou X.  Modeling diseases with latency and relapse. Mathematical Bioscience and Engineering (2009);4:205--219.

\bibitem{43}
Wang J, Zhang J, Zhen Jin.  Analysis of a SIR model with bilinear incidence rate. Nonlinear Analysis: Real World Applications (2010);11(4);2390--2402.

\bibitem{44}
Wei-min L, Levin Simon A, Iwasa Yoh.  Influence of nonlinear incidence rates upon the behaviour of SIRS epidemiological models. J. Mathematical Biology (1986);23(2):187--204.

\bibitem{45}
Yoshida N, Hara T. Global stability of a delayed SIR epidemic model with density dependent birth and death rates. J. Comput. Appl. Math. (2007);201(2):339--347.

\bibitem{46}
Zhang T, Z. Teng. Global asymptotic stability of a delayed SEIRS epidemic model with saturation incidence. Chaos, Solitons and Fractals (2008);37(5):1456--1468.

\end{thebibliography}

\end{document}